\documentclass[12pt]{amsart}
\usepackage{amsthm}

\usepackage[latin1]{inputenc}
\usepackage{subfigure}
\usepackage{color}
\usepackage{amsmath}
\usepackage{amsthm}
\usepackage{amstext}
\usepackage{amssymb}
\usepackage{amsfonts}
\usepackage{graphicx}
\usepackage{young}
\usepackage{multicol}
\usepackage{mathrsfs}
\usepackage{stmaryrd}
\usepackage{bbm}
\usepackage[all]{xy}
\usepackage{rotating}
\usepackage{mathabx}
\usepackage{tikz}
\usepackage{mathtools}
\usepackage{tabularx}
\usepackage{array}
\usepackage{commath}

\theoremstyle{plain}
\theoremstyle{definition}
\newtheorem{theorem}{Theorem}[section]

\newtheorem{remark}[theorem]{Remark}
\newtheorem{lemma}[theorem]{Lemma}
\newtheorem{definition}[theorem]{Definition}

\newtheorem{example}[theorem]{Example}
\newtheorem{proposition}[theorem]{Proposition}

\DeclareMathAlphabet{\mathpzc}{OT1}{pzc}{m}{it}

\usepackage{amssymb,amsmath,tabularx,graphicx}
\usepackage{tikz}
\usetikzlibrary[calc,intersections,through,backgrounds,arrows,decorations.pathmorphing]

\def\ra{\rightarrow}
\def\pr{\prime}

\DeclareMathOperator{\Cov}{Cov}
\DeclareMathOperator{\Bic}{Bic}

\newcommand*{\vcenteredhbox}[1]{\begingroup
\setbox0=\hbox{#1}\parbox{\wd0}{\box0}\endgroup}

\begin{document}

\title{The canonical join complex for biclosed sets}
\date{}
\author{Alexander Clifton}
\address{Department of Mathematics and Computer Science, Emory University}
\email{aclift2@emory.edu}

\author{Peter Dillery}
\address{Department of Mathematics, University of Michigan}
\email{dillery@umich.edu}

\author{Alexander Garver}
\address{Laboratoire de Combinatoire et d'Informatique Math\'ematique,
Universit\'e du Qu\'ebec \`a Montr\'eal}
\email{alexander.garver@lacim.ca}

\maketitle

\begin{abstract}
The canonical join complex of a semidistributive lattice is a simplicial complex whose faces are canonical join representations of elements of the semidistributive lattice. We give a combinatorial classification of the faces of the canonical join complex of the lattice of biclosed sets of segments supported by a tree, as introduced by the third author and McConville. We also use our classification to describe the elements of the shard intersection order of the lattice of biclosed sets. As a consequence, we prove that this shard intersection order is a lattice.
\end{abstract}

\tableofcontents

\section{Introduction}

In \cite{mcconville2017lattice}, McConville introduced a lattice of biclosed sets as a tool for studying the lattice structure of Grid-Tamari orders. The class of these lattices of biclosed sets includes the weak order on permutations. As the weak order on permutations appears in many mathematical contexts, including (and certainly not limited to) geometric combinatorics \cite{reading:cambrian_lattices} \cite{hohlweg2011permutahedra} and representation theory of preprojective algebras \cite{mizuno2014classifying, thomas2017stability}, it is natural to study the lattice-theoretic aspects of biclosed sets.

In subsequent work by McConville and the third author \cite{garver2015lattice} \cite{garver2016oriented} \cite{garver2017enumerative}, biclosed sets were used to understand the lattice structure and other lattice-theoretic questions about Grid-Tamari orders and oriented flip graphs. Futhermore, in \cite{garver2016oriented} \cite{garver2017enumerative}, the authors describe the lattice-theoretic shard intersection order, in the sense of \cite{ReadingPAB}, of the Grid-Tamari order and of oriented flip graphs. The goal of this paper is to gain a combinatorial description of this lattice-theoretic shard intersection order of the lattice of biclosed sets appearing in \cite{garver2016oriented}. 

To understand this shard intersection order, it is very useful to understand the canonical join complex of the lattice of biclosed sets. The canonical join complex is defined for any semidistributive lattice $L$.  The lattices of biclosed sets that we consider in this paper are all congruence-uniform, which implies that they are semidistributive. The canonical join complex is the simplicial complex whose faces are canonical join representations of elements of $L$. The shard intersection order of $L$, denoted $\Psi(L)$, is an alternative partial order on the elements of $L$ that is constructed using the data of canonical join representations of elements of $L$. We remark that if $L$ is not congruence-uniform, then $\Psi(L)$ may not be partially ordered.

Our approach is to, first, describe the join-irreducible biclosed sets (see Proposition~\ref{Prop_join_irr_des}). After that, we use this description to classify the faces of the canonical join complex of biclosed sets (see Theorem~\ref{Thm_CJC}). We are then in a position describe the elements of the shard intersection order (see Theorem~\ref{psib}) and the lattice structure of the shard intersection order (see Theorem~\ref{Thm_lattice_prop}). In particular, we prove that the shard intersection order of biclosed sets is a lattice.

The paper is organized as follows. We remind the reader of the lattice theory that we will use throughout the paper in Section~\ref{Sec_lattice}. We describe the lattices of biclosed sets we will work with in Section~\ref{Sec_biclosed}. In Section~\ref{cul}, we construct a special labeling of the covering relations in lattices of biclosed sets and use this labeling to index the join-irreducible and meet-irreducible biclosed sets. We then describe the faces of the canonical join complex of biclosed sets in Section~\ref{Sec_cjc}. Lastly, we study shard intersection order of biclosed sets in Section~\ref{shard_section}.

%\textcolor{blue}{mention definition of congruence uniform in terms of bijection between join-irr, meet-irr, and join-irr lattice congruences.}

\section{Preliminaries}\label{Sec_prelim}

\subsection{Lattices}\label{Sec_lattice}

%\textcolor{red}{cut EL- stuff}

Let $(L,\le_L)$ be a finite lattice. For $x,y \in L$, if $x < y$ and there does not exist $z \in L$ such that $x < z <y$, we write $x \lessdot y$. Let $\Cov(L) := \{(x,y) \in L^{2} \mid x \lessdot y\}$ be the set of \textbf{covering relations} of $L$. We let $\hat{0}, \hat{1} \in L$ denote the unique minimal and unique maximal elements of $L$, respectively. 

A set map $\lambda : \Cov(L) \to Q$, where $(Q, \leq_{Q})$ is some poset is called an \textbf{edge labeling}. We review the concepts of join- and meet-irreducibility in order to discuss an important type of labeling. %Given two maximal chains $C = c_{1} < \dots < c_{n}$ and $C' = c_{1}' < \dots c_{n}'$ in $L$, we say $C$ is \textbf{lexicographically smaller} than $C'$ if $(\lambda(c_{1}, c_{2}), \dots, \lambda(c_{n-1}, c_{n}))$ lexicographically precedes $(\lambda(c_{1}', c_{2}'), \dots, \lambda(c_{n-1}', c_{n}'))$.

%We call a labeling $\lambda : \Cov(L) \to Q$ an \textbf{EL-labeling} of $L$ if for every interval $[x,y]$ of $L$, \begin{enumerate}
%\item{there is a unique increasing maximal chain $C$ in $[x,y]$, and}
%\item{$C$ is lexicographically smaller than any other maximal chain $C'$ in $[x,y]$.}
%\end{enumerate}
%If $L$ admits an EL-labeling, it is said to be \textbf{EL-shellable}.

We say that an element $j \in L$ is \textbf{join-irreducible} if $j \neq \hat{0}$ and whenever $j = x \vee y$, either $j = x$ or $j = y$ holds. \textbf{Meet-irreducible} elements $m \in L$ are defined dually. We denote the subset of join-irreducible (resp., meet-irreducible) elements by $\text{JI}(L)$ (resp., $\text{MI}(L)$). For $j$ (resp., $m$) in $\text{JI}(L)$ (resp., $\text{MI}(L)$), we let $j_{*}$ (resp., $m^{*}$) denote the unique element of $L$ covered by (resp., that covers) $j$ (resp., $m$). 

For $A\subseteq L$, the expression $\bigvee A := \bigvee_{a \in A} a$ is \textbf{irredundant} if there does not exist a proper subset $A^{\pr}\subsetneq A$ such that $\bigvee A^{\pr}=\bigvee A$. Given $A,B\subseteq \text{JI}(L)$ such that $\bigvee A$ and $\bigvee B$ are irredundant and $\bigvee A=\bigvee B$, we set $A\preceq B$ if for $a\in A$ there exists $b\in B$ with $a\leq b$. In this situation, we say that $\bigvee A$ is a \textbf{refinement} of $\bigvee B$. If $x\in L$ and $A\subseteq\text{JI}(L)$ such that $x=\bigvee A$ is irredundant, we say $\bigvee A$ is a \textbf{canonical join representation} of $x$ if $A\preceq B$ for any other irrendundant join representation $x=\bigvee B,\ B\subseteq\text{JI}(L)$. Dually, one defines \textbf{canonical meet representations}.

We define the \textbf{canonical join complex} of $L$, denoted $\Delta^{CJ}(L)$, to be the abstract simplicial complex whose vertex set is $\text{JI}(L)$ and whose faces are sets of join-irreducibles whose join is a canonical join representation of some  element of $L$. 

Now we assume that $L$ is a \textbf{semidistributive} lattice. This means that for any three elements $x, y, z \in L$, the following properties hold:
\begin{itemize}
\item if $x \wedge z = y \wedge z$, then $(x \vee y) \wedge z = x \wedge z$, and
\item if $x \vee z = y \vee z$, then $(x \wedge y) \vee z = x \vee z$.
\end{itemize}
It is known that a lattice $L$ is semidistributive if and only if each element of $L$ has a canonical join representation and a canonical meet representation \cite[Theorem 2.24]{freese1995free}. In this case, there is a canonical bijection $L \to \Delta^{CJ}(L)$ sending $x \mapsto A$ where $\bigvee A$ is the canonical join representation of $x$.

With these notions in hand, we arrive at the notions of \textbf{CN-} and \textbf{CU-labeling}, the latter of which plays a prominent role in this paper.

\begin{definition}\label{defn_cu} A labeling  $\lambda: \Cov(L) \to Q$ is a \textbf{CN-labeling} if $L$ and its dual $L^{*}$ satisfy the following: given $x,y,z \in L$ with $(z,x), (z,y) \in \Cov(L)$ and maximal chains $C_{1}$ and $C_{2}$ in $[z, x \vee y]$ with $x \in C_{1}$ and $y \in C_{2}$, \begin{enumerate}
\item[(CN$1$)]{the elements $x' \in C_{1}, y' \in C_{2}$ such that $(x', x \vee y), (y', x \vee y) \in \Cov(L)$ satisfy $$\lambda(x', x \vee y) = \lambda(z,y), \lambda(y', x \vee y) = \lambda(z, x);$$}
\item[(CN$2$)]{if $(u,v) \in \Cov(C_{1})$ with $z < u, v < x \vee y$, then $\lambda(z,x), \lambda(z,y) <_{Q} \lambda(u,v)$;}
\item[(CN$3$)]{the labels on $\Cov(C_{1})$ are pairwise distinct.}
\end{enumerate}

We say that $\lambda$ is a \textbf{CU-labeling} if, in addition, it satisfies \begin{enumerate}
\item[(CU$1$)]{$\lambda(j_{*}, j) \neq \lambda(j_{*}', j')$ for $j, j' \in \text{JI}(L)$, $j \neq j'$, and}
\item[(CU$2$)]{$\lambda(m, m^{*}) \neq \lambda(m, m^{'*})$ for $m, m' \in \text{MI}(L)$, $m \neq m'$.}
\end{enumerate} 

If $L$ admits a CU-labeling, it is said to be \textbf{congruence-uniform}. %\textcolor{red}{CUT?} One can equivalently define a finite lattice to be congruence-uniform if it can be constructed from the lattice with a single element by a finite sequence of \textbf{interval doublings}.
\end{definition}

\begin{remark}\label{Remark_def_of_CU_lattice}
For completeness, we include the more standard definition of a congruence-uniform lattice. 

Given $(x, y) \in \text{Cov}(L)$, we let $\text{con}(x,y)$ denote the most refined lattice congruence for which $x \equiv y$. Such congruences are join-irreducible elements of the lattice of lattice congruences of $L$, denoted $\text{Con}(L)$. When $L$ is a finite lattice, the join-irreducibles (resp., meet-irreducibles) of $\text{Con}(L)$ are the congruences of the form $\text{con}(j_*,j)$ (resp., $\text{con}(m,m^*)$). We thus obtain surjections 
$$\begin{array}{rclccrcl}
\text{JI}(L) & \rightarrow & \text{JI}(\text{Con}(L)) & & &  \text{MI}(L) & \rightarrow & \text{MI}(\text{Con}(L)) \\
j & \mapsto & \text{con}(j_*,j) & & &  m & \mapsto & \text{con}(m,m^*).
\end{array}$$
If these maps are bijections, we say that $L$ is \textbf{congruence-uniform}. It follows from \cite[Proposition 2.5]{garver2016oriented} that this definition and the one given in Definition~\ref{defn_cu} are equivalent.
\end{remark}

We conclude this section by mentioning some general properties of CU-labelings and the definition of the lattice-theoretic shard intersection order of $L$. Given a edge labeling $\lambda: \text{Cov}(L) \to Q$, one defines $$\lambda_{\downarrow}(x) := \{\lambda(y,x): \ y\lessdot x\}, \ \ \ \lambda^{\uparrow}(x):= \{\lambda(x,z): \ x \lessdot z\}.$$ 

\begin{lemma}\cite[Lemma 2.6]{garver2016oriented}\label{cu-label_ji_mi}
Let $L$ be a congruence-uniform lattice with CU-labeling $\lambda:\Cov(L)\ra P$. For any $s\in P$, there is a unique join-irreducible $j \in \text{JI}(L)$ (resp., meet-irreducible $m \in \text{MI}(L)$) such that $\lambda(j_*,j) = s$ (resp., $\lambda(m,m^*) = s$). Moreover, this join-irreducible $j$ (resp., meet-irreducible $m$) is the minimal (resp., maximal) element of $L$ such that $s \in \lambda_\downarrow(j)$ (resp., $s \in \lambda^\uparrow(m)$).
\end{lemma}

%\textcolor{red}{Let $L$ be a congruence-uniform lattice with CU-labeling $\lambda:\Cov(L)\ra P$. For any $s\in P$, if $j$ is a minimal element with the property $s\in \lambda_{\downarrow}(j)$, then $j$ is a join-irreducible. Dually, if $m$ is a maximal element with the property $s\in\lambda^{\uparrow}(m)$, then $m$ is meet-irreducible.}

Later, in Propositions~\ref{Prop_join_irr_des} and \ref{Prop_meet_irr_des}, we use Lemma~\ref{cu-label_ji_mi} to characterize join- and meet-irreducible elements of $\text{Bic}(T),$ the lattice of biclosed sets defined in the next section. 

%\textcolor{red}{Later, in Propositions~\ref{Prop_join_irr_des} and \ref{Prop_meet_irr_des}, we prove the converse of the statements in Lemma~\ref{cu-label_ji_mi} when $L$ is the lattice of biclosed sets defined in the next section. }

One can also use CU-labelings to determine canonical join representations and canonical meet representations of elements of a congruence-uniform lattice. We state this precisely as follows.

\begin{lemma}\cite[Proposition 2.9]{garver2016oriented}\label{Lemma_cjr_cmr}
Let $L$ be a congruence-uniform lattice with CU-labeling $\lambda$. For any $x\in L$, the canonical join representation of $x$ is $\bigvee D$, where $D = \{j \in \text{JI}(L): \ \lambda(j_*,j) \in \lambda_\downarrow(x)\}$. Dually, for any $x\in L$, the canonical meet representation of $x$ is $\bigwedge U$, where $U = \{m \in \text{MI}(L): \ \lambda(m,m^*) \in \lambda^\uparrow(x)\}$.
\end{lemma}

%\textcolor{red}{Let $L$ be a congruence-uniform lattice with CU-labeling $\lambda$. For $x\in L$, the canonical join representation of $x$ is $\bigvee_D j$, where $D$ is the set of join-irreducibles such that $\lambda(j_*,j)\in\lambda_{\downarrow}(x)$. Dually, for $x\in L$, the canonical meet representation of $x$ is $\bigwedge_U m$, where $U$ is the set of meet-irreducibles such that $\lambda(m,m^*)\in\lambda^{\uparrow}(x)$.}

Given a lattice $L$ with a CU-labeling, one can define a new partial order on the elements of $L$ known as the \textbf{shard intersection order} of $L$. Reading introduced this concept in \cite{ReadingPAB}.

\begin{definition} Let $L$ be a congruence-uniform lattice with CU-labeling $\lambda : \Cov(L) \to P$. Let $x \in L$ and let $y_{1}, \dots y_{k}$ be the elements of $L$ satisfying $(y_{i}, x) \in \Cov(L)$. We denote the set $\{\lambda(y_{i}, x)\}$ by $\lambda_{\downarrow}(x)$. Define the \textbf{shard intersection order} of $L$, denoted $\Psi(L)$, to be the collection of sets of the form $$\psi(x) := \{ \lambda(w,z) \mid \bigwedge_{i}^{k} y_{i} \leq w < z \leq x, (w,z) \in \Cov(L)\}$$ partially ordered by inclusion. At times, we may refer to the interval $[\bigwedge_{i}^{k} y_{i},x]$ as a \textbf{facial interval}.
\end{definition}

\begin{remark}\label{rel_to_hyperplanes}
The shard intersection order was originally defined by Reading in \cite{reading:noncrossing} when $L = \text{Pos}(\mathcal{A}, B)$ is the poset of regions of a simplicial hyperplane arrangement $\mathcal{A}$ with base region $B$. If, in addtion, $L$ is congruence-uniform, it follows from \cite[Proposition 9-7.13]{ReadingPAB} that these two definitions of the shard intersection order produce the same partial order.
\end{remark}

\subsection{Biclosed sets}\label{Sec_biclosed}

A \textbf{tree} is a finite connected acyclic graph. The degree-one vertices of a tree are called \textbf{leaves}. We can always embed a tree $T$ into the disk $D^{2}$ so that exactly the leaves lie on the boundary. Unless stated otherwise, a tree is assumed to be equipped with such an embedding. Non-leaf vertices of $T$ are thus in the interior of $D^{2}$, and we call these \textbf{interior vertices}. We also assume that the interior vertices of $T$ have degree at least 3.

An \textbf{acyclic path} is a sequence of pairwise distinct vertices $(v_{i_{1}}, \dots, v_{i_{n}})$ of $T$ such that there is an edge connecting $v_{i_{j}}$ and $v_{i_{j+1}}$ for all $1 \leq j \leq n-1$. Since an acyclic path is uniquely determined by its endpoints, we can denote the path $(v_{i_{1}}, \dots, v_{i_{n}})$ by $[v_{i_{1}}, v_{i_{n}}]$. 

%Consider two vertices $v$ and $w$ in $T$ connected by an edge $[v,w]$. \textcolor{red}{Orienting $[v,w]$ from $v$ to $w$ gives rise to a cyclic orientation of the edges incident to $w$.} 

Observe that the embedding of $T$ in $D^2$ determines are cyclic ordering of the edges of $T$ that are incident to a given vertex. An acyclic path $(v_{i_{1}}, \dots, v_{i_{n}})$ is called a \textbf{segment} if, for each $1 \leq j \leq n-2$, $[v_{i_{j+1}}, v_{i_{j+2}}]$ is immediately clockwise or counterclockwise from $[v_{i_{j}}, v_{i_{j+1}}]$ with respect to the cyclic ordering on the edges incident to $v_{i_{j+1}}$. The set of all segments supported by a tree $T$ is denoted by $\text{Seg}(T)$. Figure~\ref{seg_fig} shows some examples and non-examples of segments.

\begin{figure}
\includegraphics[scale=1.5]{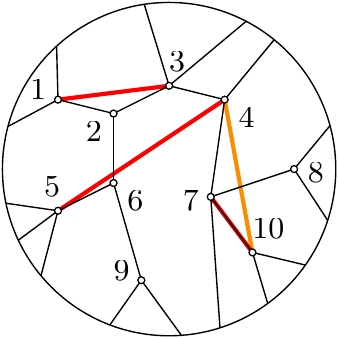}
\caption{Acyclic paths $[1,3], [5,4],$ and $[7,10]$ are segments, but acyclic path $[4,10]$ is not.}
\label{seg_fig}
\end{figure}

Given two segments $s_{1} = (v_{i_1}, v_{i_2}, \ldots, v_{i_k}), s_{2} = (v_{i_k}, v_{i_{k+1}}, \ldots, v_{i_{n}}) \in \text{Seg}(T)$ that share an endpoint $v_{i_k}$ but differ at all other vertices, we define their \textbf{composition} to be the acyclic path $s_{1} \circ s_{2} := (v_{i_1}, \ldots, v_{i_k}, \ldots, v_{i_n})$. We say that two segments $s_1$ and $s_2$ are \textbf{composable} if $s_1 \circ s_2 \in \text{Seg}(T)$. A subset $B\subset \text{Seg}(T)$ is \textbf{closed} if for all composable $s_1, s_2 \in B, s_1 \circ s_2 \in B$. $B$ is \textbf{biclosed} if both $B$ and its complement, $B^c := \text{Seg}(T)\backslash B$, are closed. We will also often say that $B$ is \textbf{coclosed} when $B^c$ is closed. Additionally, if $B \subset \text{Seg}(T)$, we define $\overline{B}$ to be the smallest closed set containing $B$.

The poset structure of $\text{Bic}(T)$ is studied in \cite{garver2016oriented}, where the following result is proved:

\begin{theorem}\cite[Theorem 4.1]{garver2016oriented}\label{bic_structure} The poset $\text{Bic}(T)$ is a semidistributive, congruence-uniform, and polygonal lattice. Moreover, the $\text{Bic}(T)$ has the following properties:
\begin{enumerate}
\item for any $X, Y \in \text{Bic}(T)$, if $X \subsetneq Y$, then there is a segment $y \in Y$ such that $X \sqcup \{y\} \in \text{Bic}(T);$
\item for any $W, X, Y \in \text{Bic}(T)$ with $W \le X\cap Y$, the set $W \cup \overline{(X\cup Y)\backslash W}$ is biclosed;
\item the edge-labeling $\lambda: \text{Cov}(\text{Bic}(T)) \to \text{Seg}(T)$ defined by $\lambda(X,Y) = s$ if $Y\backslash X = \{s\}$ is a CN-labeling.
\end{enumerate}
\end{theorem}

\begin{remark}\label{joinmeet_bijection} Theorem~\ref{bic_structure} implies that the map $(-)^{c} \colon \text{Bic}(T) \to \text{Bic}(T)$, $B \mapsto B^{c}$, gives rise to a bijection $\text{JI}(\text{Bic}(T)) \to \text{MI}(\text{Bic}(T))$. 
\end{remark}

We frequently use the following lemma. It follows from property (2) in Theorem~\ref{bic_structure} with $X = B_1,$ $Y = B_2$, and $W = \emptyset$.

\begin{lemma}\label{Join_Description}
For any $B_{1}, B_{2} \in \text{Bic}(T)$, we have $B_{1} \vee B_{2} = \overline{B_{1} \cup B_{2}}.$
\end{lemma}

We use the next lemma to calculate facial intervals in $\text{Bic}(T)$ (see Lemma~\ref{Lemma_facial_int_atoms}).

\begin{lemma}\label{Lemma_meet_in_bic}
For $B_1, B_2 \in \text{Bic}(T)$, one has $B_1 \wedge B_2 = (B_1^c \vee B_2^c)^c.$
\end{lemma}
\begin{proof}
Let $s \in B_1 \wedge B_2$. It follows that $s \in B_1 \cap B_2$. Now $s \not \in B_1^c\cup B_2^c \subset {B_1^c \vee B_2^c}.$ This implies that $s \in (B_1^c \vee B_2^c)^c.$

To prove the opposite inclusion, observe that $$B_1 \wedge B_2 = \displaystyle \bigvee_{\begin{array}{c}\small\text{$B \in \text{Bic}(T)$}\\ \small \text{$B \subset B_1, B\subset B_2$} \end{array}} B.$$ Now notice that if $s \in (B_1^c \vee B_2^c)^c$, then $s \not \in B_1^c$ and $s \not \in B_2^c$. This implies that $s \in B_1$ and $s \in B_2$. Thus $(B_1^c \vee B_2^c)^c \subset B_1$ and $(B_1^c \vee B_2^c)^c \subset B_2$. Since $(B_1^c \vee B_2^c)^c \in \text{Bic}(T)$, it follows that $(B_1^c \vee B_2^c)^c$ is a joinand in the join representation of $B_1 \wedge B_2$ shown above. We obtain that $(B_1^c \vee B_2^c)^c \subset B_1 \wedge B_2$.\end{proof}

%\begin{proof} 
%\end{proof}

%It suffices to show that $\overline{B_{1} \cup B_{2}}$ is biclosed, since it must be contained in $B_{1} \vee B_{2}$. It's closed by construction, so we just need to show that it's coclosed. Suppose by contradiction that there exist $s_{a}, s_{b} \in (B_1\cup B_2)^{c}$ such that $s_{a} \circ s_{b} \in B_1\cup B_2$, say $s_{a} \circ s_{b} = s_{1} \circ s_{2} \circ s_{3} \dots s_{m}, 	m \in \mathbb{N}$, where we assume without loss of generality that $s_{1} \in B_{1},$ and $s_{j} \in B_{i}$ if $j \equiv i \mod 2$. We can also assume that the endpoint of $s_{a} \circ s_{b}$ that $s_{a}$ contains lies in $s_{1}$. Let $i$ be the smallest index such that $s_{i}$ is not a subsegment of $s_{a}$. Then $s_{a} = s_{1} \circ \dots \circ s_{i-1} \circ s_{i}', s_{b} = s_{i}'' \circ s_{i+1} \circ \dots \circ s_{m}$, where $s_{i}', s_{i}'' \notin B_{1} \cup B_{2}$ and $s_{i}' \circ s_{i}'' = s_{i}$. This contradicts the fact that $B_{1}$ and $B_{2}$ are biclosed.

\begin{example}
Let $T$ be the tree shown in Figure~\ref{perm_bic_fig} with the indicated labeling of the interior vertices. Define a map that sends a segment $s$ to $(i,j) \in \mathbb{N}^2$ with $i < j$ where $i$ and $j$ are the vertex labels of the endpoints of $s$. This induces a map on biclosed sets that sends each biclosed set to the inversion set of a permutation in $\mathfrak{S}_n$. Moreover, it induces a poset isomorphism $\text{Bic}(T) \to \text{Weak}({\mathfrak{S}_n})$ where the latter denotes the weak order on permutations.

Additionally, it follows from \cite[Theorem 10-3.1]{ReadingPAB} that $\text{Weak}(\mathfrak{S}_n)$ is isomorphic to $\text{Pos}(\mathcal{A}, B)$ where $\mathcal{A}$ is Coxeter arrangement of $\mathfrak{S}_n$ and $B$ is the region of the Coxeter arragement containing the identity permutation. Now it follows from Remark~\ref{rel_to_hyperplanes} that the class of lattices of the form $\Psi(\text{Bic}(T))$ includes the shard intersection orders of type $A$ Coxeter arrangements.
\end{example}

\begin{figure}
\centering
\includegraphics[scale=1]{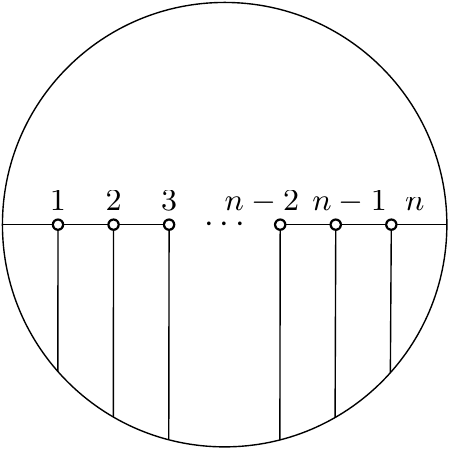}
\caption{A tree whose biclosed sets are identified with permutations in $\mathfrak{S}_n$.}
\label{perm_bic_fig}
\end{figure}

%\begin{remark}\label{joinmeet_bijection} \textcolor{red}{This result means that the map $(-)^{c} \colon \text{Bic}(T) \to \text{Bic}(T)$, $B \mapsto B^{c}$, gives rise to a bijection $\text{JI}(\text{Bic}(T)) \to \text{MI}(\text{Bic}(T))$. Suppose that $B$ is join-irreducible, which means that there exists a unique biclosed $B_*$ satisfying $B_* \lessdot B$, and $B_* \cup \{s\} = B$ for some segment $s \in B$, by Theorem \ref{bic_structure}. Again by Theorem \ref{bic_structure}, if $B^{c} \lessdot A$, then $A =  B^{c} \cup \{s\}$, which means that $A^{c} \cup \{s\} = B$, and so $A$ is unique, and hence $B^{c}$ is meet-irreducible. The map $(-)^{c}$ is now immediately a bijection since it's an involution.}
%\end{remark}

%\textcolor{purple}{Trees; segments; define composition; biclosed sets; recall formula for join; state theorem 4.1 from 1604; mention that 4.1 implies that $(-)^c$ induces a bijection between join- and meet-irreducibles; use bijection to give alternative proof of (CU1) property in proof of Theorem~\ref{cuproof} below.}

Using the following lemma, one obtains an explicit description of the facial intervals of $\text{Bic}(T)$. One can find a proof of the following lemma, in \cite[Lemma 4.14]{garver2017enumerative}.

\begin{lemma}\label{meet_of_B_is}
Given $B, B\backslash\{s_1\}, \ldots, B\backslash\{s_k\} \in \text{Bic}(T)$, we have that $\bigwedge_{i = 1}^k B_i = B\backslash\overline{\{s_1,\ldots, s_k\}}.$ %where $s_i = \lambda(B_i, B)$ for $i = 1, \ldots, k.$
\end{lemma}

\section{A CU-labeling of $\Bic(T)$}\label{cul}
In \cite{garver2016oriented}, the authors prove that $\text{Bic}(T)$ is congruence-uniform, and thus it admits a CU-labeling. In this section, we explicitly construct such a labeling.

We say a segment $s \in \text{Seg}(T)$ is a \textbf{split} of a segment $t$ if $s$ is a proper subsegment of $t$, and $s$ and $t$ share an endpoint. A \textbf{break} of a segment $[a,c]$ is a pair of splits of $[a,c]$, denoted $\{[a,b],[b,c]\}$, for some vertex $b$ of segment $[a,c]$ lying between $a$ and $c$. We say that $b$ is the \textbf{faultline} of the break $\{[a,b],[b,c]\}$.

Define a poset $\mathcal{S}_T$ whose elements are of the form $(s,{\{s_1,s_2,\ldots,s_m\}}) \in \text{Seg}(T)\times 2^{\text{Seg}(T)}$ with the following properties:
\begin{itemize}
\item $s = (v_0, v_1,\ldots, v_{m+1})$ has $m$ breaks,
\item each $s_i$ is a split of $s$, and
\item two distinct splits $s_i$ and $s_j$ do not appear in the same break of $s$. 
\end{itemize}
We will typically denote $(s,{\{s_1,s_2,\ldots,s_m\}}) \in \mathcal{S}_T$ by $s_\mathcal{D} = s_{\{s_1,s_2,\ldots,s_m\}}$. The elements of $\mathcal{S}_T$ are partially ordered as follows: given $s_{\{s_1,s_2,\ldots,s_k\}}, {s^\prime_{\{t_1,t_2,\ldots,t_l\}}} \in \mathcal{S}_T$, one has ${s^\prime_{\{t_1,t_2,\ldots,t_l\}}} \le s_{\{s_1,s_2,\ldots,s_k\}}$ if $s^\prime$ is a proper subsegment of $s$. At times, we will also write $s^\prime \subseteq s$ (resp., $s^\prime \not \subseteq s$) to indicate that $s^\prime$ is a subsegment (resp., is not a subsegment) of $s$. We will refer to elements of $\mathcal{S}_T$ as \textbf{labels}.

\begin{example}
Let $T$ be the tree shown in Figure~\ref{poset_of_labels_fig}. In this same figure, we show the poset of labels $\mathcal{S}_T$ of the covering relations of $\text{Bic}(T)$.
\end{example}

\begin{figure}
\includegraphics[scale=1]{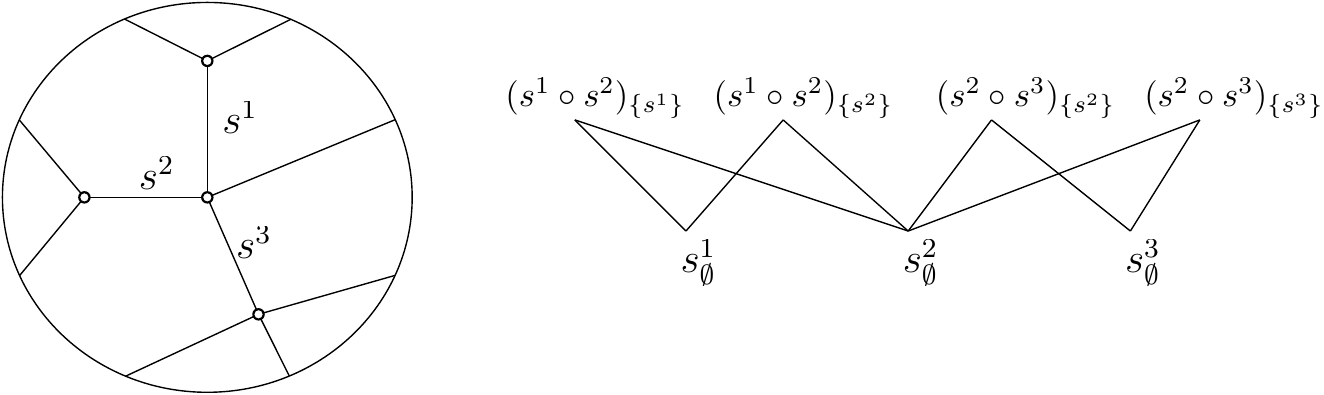}
\caption{A tree $T$ and the corresponding poset of labels $\mathcal{S}_T$. The shortest segments of $T$ are labeled $s^1, s^2,$ and $s^3$.}
\label{poset_of_labels_fig}
\end{figure}

\begin{definition}
Define a map $\widetilde{\lambda}:\text{Cov}(\text{Bic}(T))\rightarrow \mathcal{S}_T$ by
 $\widetilde{\lambda}(B,B\sqcup\{{s}\})=s_{\{s_1,s_2,\ldots,s_k\}}$ where $s_1,s_2,\ldots,s_k$ are the splits of $s$ which are contained in $B$. It is clear that $\widetilde{\lambda}$ is an edge-labeling of $\text{Bic}(T)$. If we let $\lambda: \text{Cov}(\text{Bic}(T)) \to \text{Seg}(T)$ denote the first coordinate function of $\widetilde{\lambda}$, we have that $\lambda$ is the CN-labeling of $\text{Bic}(T)$ from Theorem~\ref{bic_structure} (3).
\end{definition}

\begin{example}
Let $B$ denote the biclosed set shown in Figure~\ref{label_fig}. One checks that $$\widetilde{\lambda}(B\backslash\{[9,8]\}, B) = [9,8]_{\{[9,2], [9,4], [6,8], [3,8], [7,8]\}}.$$
\end{example}

%Observe that any two labels $\widetilde{\lambda}(B,B\cup\{{[a,c]}\})=[a,c]_{\{s_1,s_2,\ldots,s_k\}}$ and $\widetilde{\lambda}(B^\prime,B^\prime\cup\{{[a,c]}\})= [a,c]_{\{t_1,t_2,\ldots,t_l\}}$ have $|\{s_1,s_2,\ldots,s_k\}| = |\{t_1,t_2,\ldots,t_l\}|$, since $B\cup\{{[a,c]}\}$ and $B^\prime\cup\{{[a,c]}\}$ are biclosed sets. Moreover, 

%\textcolor{red}{NEED GOOD EXAMPLE WITH PICTURE}

\begin{figure}
\includegraphics[scale=1.5]{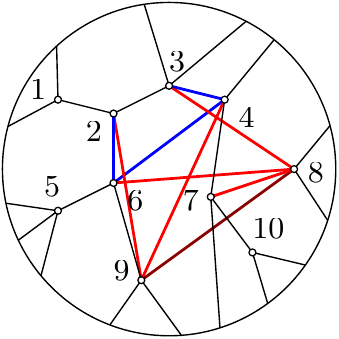}
\caption{The join irreducible biclosed set $J([9,8]_{\{[9,2], [9,4], [6,8], [3,8], [7,8]\}})$ where the tree shown here is the same tree from Figure~\ref{seg_fig}.}
\label{label_fig}
\end{figure}

\begin{theorem}\label{cuproof}
The edge-labeling $\widetilde{\lambda}:\text{Cov}(\text{Bic}(T))\rightarrow \mathcal{S}_T$ is a CU-labeling of $\text{Bic}(T)$.
\end{theorem}

%\begin{example}
%	Consider the tree $T=P_4$, consisting of a path of 4 vertices numbered 1,2,3,4 in order. Consider the biclosed sets $B=\{[1,2],[1,3]\}$ and $C=\{[1,2],[1,3],[1,4]\}$.\\
%	
%	Then $\widetilde{\lambda}(B\lessdot C)=[1,4]_{\{[1,2],[1,3]\}}$.
%	
%\end{example}

\begin{proof}[Proof of Theorem \ref{cuproof}]
We first verify axioms (CN1), (CN2), and (CN3). Note that it is enough verify these axioms for $\text{Bic}(T)$ since $\text{Bic}(T)$ is self-dual. 

Let $B_1, B_2 \in \text{Bic}(T)$ and assume that they both cover some biclosed set $B$. By Theorem~\ref{bic_structure}, there exists $s, t \in \text{Seg}(T)$ such that $B_1 = B \sqcup \{s\}$ and $B_2 = B\sqcup \{t\}$. The interval $[B, B_1\vee B_2]$ has one of the two forms shown in Figure~\ref{polygons_fig} where the left (resp., right) figure occurs when $s$ and $t$ are not composable (resp., when $s$ and $t$ are composable). One deduces axioms (CN1), (CN2), and (CN3) from Figure~\ref{polygons_fig}.
	
We now verify axioms (CU1) and (CU2).
	
	(CU2): Consider two meet-irreducibles $M_1, M_2 \in \text{JI}(\text{Bic}(T))$ which are covered by $M_1^*$ and $M_2^*$, respectively. Assume for the sake of contradiction that $\widetilde{\lambda}(M_1, M_1^*)=\widetilde{\lambda}(M_2, M_2^*),$ and denote this label by $s_\mathcal{D}$. Thus $M_1^*=M_1\sqcup\{{s}\}$ and $M_2^*=M_2\sqcup\{{s}\}$. Note that $s \in M_1\vee M_2$ so there exists $s_1, \ldots, s_k \in M_1 \cup M_2$ such that $s = s_1\circ \cdots \circ s_k$. Without loss of generality, assume that $s_1 \in M_1,$ and that $s_{i} \in M_2$ if $s_{i-1} \in M_1$ and $s_i \in M_1$ if $s_{i-1} \in M_2$ for each $i = 1, \ldots, k$.

Next, since $\widetilde{\lambda}(M_1, M_1^*)=\widetilde{\lambda}(M_2, M_2^*)$, sets $M_1$ and $M_2$ both contain the same split $s^\prime$ or $s^{\prime\prime}$ from a given break $\{s^\prime, s^{\prime\prime}\}$. We know that $s_1$ is a split of $s$ so $s_1 \in M_1 \cap M_2$. Since $s_2 \in M_2$, we know $s_1 \circ s_2 \in M_2$. Now $s_1\circ s_2$ is a split of $s$ so it follows that $s_1\circ s_2 \in M_1\cap M_2$. By continuing this argument, we obtain that $s \in M_1$, a contradiction.
	
	(CU1): Consider elements $J^1$ and $J^2$ of Bic($T$) which cover $J^1_*$ and $J^2_*$, respectively. Notice that $(J^1)^c$ and $(J^2)^c$ are meet-irreducibles of $\text{Bic}(T)$ which are covered by $(J^1_*)^c$ and $(J^2_*)^c$, respectively. If $\widetilde{\lambda}(J^1_*, J^1)=\widetilde{\lambda}(J^2_*, J^2)$ and we denote this label by $s_\mathcal{D}$, then $\widetilde{\lambda}((J^1)^c, (J^1_*)^c)=\widetilde{\lambda}((J^2)^c, (J^2_*)^c)$ and the common label is $s_{\mathcal{D}^\prime}$ where $\mathcal{D}\cap \mathcal{D}^\prime = \emptyset.$ This contradicts property (CU1).
\end{proof}

\begin{figure}
%$$\begin{array}{ccccccccccccc}
\vcenteredhbox{\includegraphics[scale=1.2]{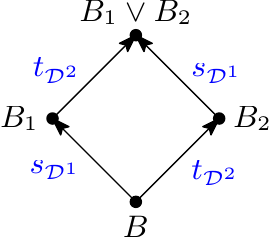}} \ \ \ \ \ \ \vcenteredhbox{\includegraphics[scale=1.2]{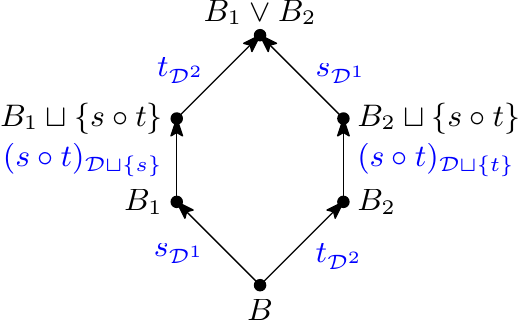}}
%\end{array}$$
\caption{{The two forms of the interval $[B,B_1\vee B_2]$ of $\text{Bic}(T)$. The labels on the covering relations as defined by the labeling $\widetilde{\lambda}: \text{Cov}(\text{Bic}(T)) \to \mathcal{S}_T$ are in blue. The set $\mathcal{D}^1$ (resp., $\mathcal{D}^2$) consists of all splits of $s$ (resp., $t$) belonging to $B$. Similarly, the set $\mathcal{D}$ consists of all splits of $s\circ t$ that belong to $B$.}}
\label{polygons_fig}
\end{figure}

%\begin{remark} One could argue (CU1) differently; by Remark \ref{joinmeet_bijection}, it follows that for any join-irreducible $J \in \text{Bic}(T)$, the set of labels $\{\lambda(J, J^{*}) \mid J \lessdot J^{*}\}$ is the same as the same as $\{\lambda(M^{*}, M) \mid M^{*} \lessdot M\}$ for some unique meet-irreducible $M$, and so (CU1) follows immediately from (CU2).

%\end{remark}

We conclude this section by classifying the join- and meet-irreducible biclosed sets. We do this by choosing a label $s_\mathcal{D} \in \mathcal{S}_T$ and constructing the minimal biclosed set $B$ where $\widetilde{\lambda}_\downarrow(B) = \{s_\mathcal{D}\}$ and by constructing the maximal biclosed set where $\widetilde{\lambda}^\uparrow(B) = \{s_\mathcal{D}\}.$

Given $s_\mathcal{D} \in \mathcal{S}_T$, define $$J(s_\mathcal{D}) := \{s\} \sqcup \mathcal{D} \sqcup \bigcup_{t \in \mathcal{D}} S(t)$$ where $S(t) = S(t,\mathcal{D}) \subset \text{Seg}(T)$ is defined to be the set of all splits $s^\prime$ of $t$ satisfying the following:
\begin{itemize}
\item[i)] segment $s^\prime$ is not a split of $s$, and
\item[ii)] segment $s^\prime$ is not composable with any segment in $\mathcal{D}$.
\end{itemize} 

\begin{example}
We give an example of a set $J(s_\mathcal{D})$ in Figure~\ref{label_fig}. Here we have that
$$\begin{array}{rcl}
s & = & [9,8] \\
\mathcal{D} & = & \{[9,2], [9,4], [6,8], [3,8], [7,8]\} \\
S([9,2]) & = & \{[6,2]\} \\
S([9,4]) & = & \{[6,4], [3,4]\} \\
S([6,8]) & = & \{[6,2], [6,4]\} \\
S([3,8]) & = & \{[3,4]\}\\
S([7,8]) & = & \emptyset.
\end{array}$$
\end{example} 

\begin{lemma}\label{Lemma_J_biclosed}
The set $J(s_\mathcal{D})$ is biclosed. In fact, no two elements of $J(s_\mathcal{D})$ are composable.
\end{lemma}
\begin{proof}
We show that $J(s_\mathcal{D})$ is closed by showing that no two elements are composable. For any element $t \in \mathcal{D}$, no two elements of $\{s\}\sqcup \mathcal{D} \sqcup S(t)$ are composable. Thus it is enough to show that given $t_1, t_2 \in \mathcal{D},$ any two segments $s_1 \in S(t_1)$ and $s_2 \in S(t_2)$ are not composable.  

Suppose that $t_1$ and $t_2$ do not agree along a segment. If $s_1$ and $s_2$ are composable, then so are $t_1$ and $t_2$. This contradicts the definition of $\mathcal{D}$.

Now suppose that $t_1$ and $t_2$ agree along a segment. If $s_1$ and $s_2$ are composable, the definition of $\mathcal{D}$ implies that there exists $d \in \mathcal{D}$ where either $s_1$ and $d$ are composable or $s_2$ and $d$ are composable. This contradicts the definition of $S(t_1)$ or that of $S(t_2)$. We conclude that $J(s_\mathcal{D})$ is closed.

Next, we show that $J(s_\mathcal{D})$ is coclosed by showing that given any element $s_1 \in J(s_\mathcal{D})$ and any break $\{s^\prime, s^{\prime\prime}\}$ of $s$, one has $s^\prime\in J(s_\mathcal{D})$ or $s^{\prime\prime} \in J(s_\mathcal{D})$. It is enough to show that this is true for elements of $J(s_\mathcal{D})\backslash\{s\}$.

Assume $s_1 \in \mathcal{D}$ and $s_1 = s^\prime\circ s^{\prime\prime}$ where $s^{\prime}$ is not a split of $s$. Since $s_1 \in \mathcal{D}$, the definition of $\mathcal{D}$ implies that there is no split of $s$ that belongs to $\mathcal{D}$ and is composable with $s^\prime$. Thus either $s^{\prime\prime} \in \mathcal{D}$ or no element of $\mathcal{D}$ is composable with $s^\prime$. Therefore, either $s^{\prime\prime} \in \mathcal{D} \subset J(s_\mathcal{D})$ or $s^\prime \in S(s_1, \mathcal{D}) \subset J(s_\mathcal{D})$.

%Assume $s_1 \in \mathcal{D}$ and $s_1 = s^\prime\circ s^{\prime\prime}$ where $s^{\prime}$ is not a split of $s$. This implies that either $s^{\prime\prime} \in \mathcal{D}$ or $s^{\prime\prime} \not \in J(s_\mathcal{D})$. We therefore assume that $s^{\prime\prime} \not \in J(s_\mathcal{D})$, otherwise we are done. Since $s_1 \in \mathcal{D}$, the definition of $\mathcal{D}$ implies that there is no split of $s$ that belongs to $\mathcal{D}$ and is composable with $s^\prime$. Thus $s^\prime \in S(s_1)$ so $s^\prime \in J(s_\mathcal{D})$.

Next, assume that $s_1 \in S(t_1)$ for some $t_1 \in \mathcal{D}$ and $s_1 = s^\prime\circ s^{\prime\prime}$ where $s^{\prime}$ is not a split of $t_1$. This implies that either $s^{\prime\prime} \in J(s_\mathcal{D})\backslash(\{s\} \sqcup \mathcal{D})$ or $s^{\prime\prime} \not \in J(s_{\mathcal{D}})$. We therefore assume $s^{\prime\prime} \not \in J(s_\mathcal{D})$. This implies that $s^{\prime\prime} \not \in S(t_1)$ so there exists $t^{\prime\prime} \in \mathcal{D}$ such that $s^{\prime\prime}\circ t^{\prime\prime} \in \text{Seg}(T)$. It follows that exactly one of the following holds:
\begin{itemize}
\item segments $s^\prime$ and $t^{\prime\prime}$ {have no common vertices}, or
\item segment $s^\prime$ is a split of $t^{\prime\prime}$.
\end{itemize}
The former case is impossible because $t_1 \in \mathcal{D}$ and $s = t_1 \circ t^{\prime\prime}$. 

We show that $s^\prime \in S(t^{\prime\prime})$. Suppose there exists $t \in \mathcal{D}$ such that $s^\prime$ and $t$ are composable. Either segments $t$ and $t^{\prime\prime}$ are composable or segments $t$ and $s_1$ are composable. Since $t^{\prime\prime} \in \mathcal{D}$ and $s_1 \in S(t_1)$, we obtain a contradiction. Thus $J(s_{\mathcal{D}})$ is coclosed. 
\end{proof}

\begin{proposition}\label{Prop_join_irr_des}
The set $J(s_\mathcal{D})$ satisfies $\widetilde{\lambda}_\downarrow(J(s_\mathcal{D})) = \{s_\mathcal{D}\}.$ Moreover, any biclosed set $B$ with $s_\mathcal{D} \in \widetilde{\lambda}_\downarrow(B)$ satisfies $J(s_\mathcal{D}) \le B$, and the reverse containment holds if and only if $\widetilde{\lambda}_{\downarrow}(B) = \{s_\mathcal{D}\}$. Consequently, the set map $J: \mathcal{S}_T \to \text{JI}(\text{Bic}(T))$ is a bijection.%In particular, any label in $\mathcal{S}_T$ appears on some covering relation of $\text{Bic}(T)$.
\end{proposition}
\begin{proof}
By Lemma~\ref{Lemma_J_biclosed}, $J(s_\mathcal{D})\backslash\{s\}$ is biclosed so $s_\mathcal{D} \in \widetilde{\lambda}_\downarrow(J(s_\mathcal{D}))$. Also, the set $J(s_\mathcal{D})\backslash \{t\}$ is not coclosed for any $t \neq s$ so $\widetilde{\lambda}_\downarrow(J(s_\mathcal{D})) = \{s_\mathcal{D}\}$. Thus $J(s_{\mathcal{D}}) \in \text{JI}(\text{Bic}(T)).$

Let $B \in \text{Bic}(T)$ be such that $s_\mathcal{D} \in \widetilde{\lambda}_\downarrow(B).$ This implies that $\{s\}\sqcup \mathcal{D} \subset B$.

Now let $s_1 \in S(t)$ for some $t \in \mathcal{D}$ and suppose that $s_1\not \in B$. Since $B \in \text{Bic}(T)$, there exists $s_1^\prime \in B$ such that $t = s_1\circ s_1^\prime$. Observe that $s_1^\prime \not \in \mathcal{D}$, but $s_1^\prime$ is a split of $s$. This means that the set of splits of $s$ that belong to $B$ is strictly larger than $\mathcal{D}$. This contradicts that $s_\mathcal{D} \in \widetilde{\lambda}_\downarrow(B).$ We conclude that $J(s_\mathcal{D}) \le B.$

The remaining assertions now follow from Lemma~\ref{cu-label_ji_mi}.\end{proof}

Next, we classify the meet-irreducible biclosed sets. Given $s_\mathcal{D} \in \mathcal{S}_T$, define $$M(s_\mathcal{D}) := J(s_\mathcal{D})\backslash\{s\} \sqcup \{t \in \text{Seg}(T): \ \text{$t \not \subseteq s$}\}\sqcup \bigcup_{t \in \mathcal{D}}R(t)$$ where $R(t) = R(t,\mathcal{D}) \subset \text{Seg}(T)$ is defined to be the set of all splits $s^\prime$ of $t$ satisfying the following:
\begin{itemize}
\item[i)] segment $s^\prime$ is not a split of $s$, and 
\item[ii)] segment $s^\prime$ is composable with some element of $\mathcal{D}$. 
\end{itemize}
Observe that if $s^\prime \in R(t)$ for some $t \in \mathcal{D}$, then there is necessarily a unique element $t^\prime \in \mathcal{D}$ where $t^\prime \subseteq t$, $t^\prime$ is composable with $s^\prime$, and $s^\prime\circ t^\prime = t$.

\begin{proposition}\label{Prop_meet_irr_des}
The set $M(s_\mathcal{D})$ satisfies $\widetilde{\lambda}^\uparrow(M(s_\mathcal{D})) = \{s_\mathcal{D}\}.$ Moreover, any biclosed set $B$ with $s_\mathcal{D} \in \widetilde{\lambda}^\uparrow(B)$ satisfies $B \le M(s_\mathcal{D})$, and the reverse containment holds if and only if $\widetilde{\lambda}^{\uparrow}(B) = \{s_\mathcal{D}\}$. Consequently, the set map $M: \mathcal{S}_T \to \text{MI}(\text{Bic}(T))$ is a bijection.
\end{proposition}
\begin{proof}
We show that $J(s_{\mathcal{D}^\prime})^c = M(s_\mathcal{D})$ where $t^\prime \in \mathcal{D}^\prime$ if and only if $s = t\circ t^\prime$ for some $t \in \mathcal{D}$. Observe that $\widetilde{\lambda}^{\uparrow}(J(s_{\mathcal{D}^\prime})^c) = \{s_{\mathcal{D}}\}$. By Remark~\ref{joinmeet_bijection}, the remaining assertions may then be concluded from Proposition~\ref{Prop_join_irr_des}.

Let $t_1 \in J(s_{\mathcal{D}^\prime})^c$ and we show that $t_1 \in M(s_\mathcal{D})$. By the defintition of $J(s_{\mathcal{D}^\prime})$, it is enough to assume that $t_1 \subseteq s$ and that $t_1$ is not a split of $s$. 

We first show that $t_1$ is a split of some segment $t \in \mathcal{D}$. Suppose $t_1$ is not a split of any segment in $\mathcal{D}$. Then there exists $t_2, t_3 \in \mathcal{D}$ such that $t_2\circ t_1, t_1\circ t_3 \in J(s_{\mathcal{D}^\prime})^c$, and both of $t_2\circ t_1$ and $t_1\circ t_3$ are splits of $s$. This shows that the set of splits of $s$ belonging to $J(s_{\mathcal{D}^\prime})^c$ is strictly larger than $\mathcal{D}$. This contradicts that $\widetilde{\lambda}^\uparrow(J(s_{\mathcal{D}^\prime})^c) = \{s_{\mathcal{D}}\}$.

Now let $t \in \mathcal{D}$ be such that $t_1$ is a split of $t$. If there exists $t_1^\prime \in \mathcal{D}^\prime$ such that $t_1$ is a split of $t_1^\prime$, then there exists $t_2 \in \mathcal{D}$ such that $s = t_1^\prime \circ t_2$ and $t_1$ is composable with $t_2$. Thus $t_1 \in R(t,\mathcal{D})$. If no such segment $t_1^\prime \in \mathcal{D}^\prime$ exists, then $t_1 \in J(s_\mathcal{D})\backslash\{s\}$. In either case, $t_1 \in M(s_\mathcal{D})$.

Conversely, assume that $t_1 \in M(s_{\mathcal{D}})$. It is clear that $J(s_\mathcal{D})\backslash\{s\} \sqcup \{t \in \text{Seg}(T): \ t \not \subseteq s\} \subset J(s_{\mathcal{D}^\prime})^c$ so we assume that $t_1 \in R(t, \mathcal{D})$ for some $t \in \mathcal{D}$. Let $t_2 \in \mathcal{D}$ be the unique split of $s$ that is composable with $t_1$. Now note that there is a unique split $t_2^\prime \in \mathcal{D}^\prime$ with the property that $t_2^\prime \circ t_2 = s$ and $t_1$ is a split of $t_2^\prime$. 

We claim that $t_1 \not \in S(t_2^\prime, \mathcal{D}^\prime)$. To see this, observe that since $t \in \mathcal{D}$, the unique segment $t^\prime$ with property that $t^\prime \circ t = s$ belongs to $\mathcal{D}^\prime$. Moreover, $t_1$ and $t^\prime$ are composable since $t_1$ is a split of $t$. Since $t_1 \not \in S(t_2^\prime, \mathcal{D}^\prime)$, we have that $t_1 \in J(s_{\mathcal{D}^\prime})^c.$\end{proof}

\section{Canonical join complex of $\text{Bic}(T)$}\label{Sec_cjc}

In this section, we describe the faces of the canonical join complex of $\text{Bic}(T)$. In \cite[Theorem 1.1]{barnard2016canonical}, it is shown that the canonical join complex of a finite semidistributive lattice $L$ is a \textbf{flag complex}. That is, the minimal nonfaces of $\Delta^{CJ}(L)$ have size two. Thus, it is enough to classify the pairs of elements of $\text{JI}(\text{Bic}(T))$ that join canonically.

\begin{theorem}\label{Thm_CJC}
A collection $\{J(s^1_{\mathcal{D}^1}), \ldots, J(s^k_{\mathcal{D}^k})\} \subset \text{JI}(\text{Bic}(T))$ is a face of $\Delta^{CJ}(\text{Bic}(T))$ if and only if labels $s^i_{\mathcal{D}^i}$ and $s^j_{\mathcal{D}^j}$ satisfy the following:
\begin{itemize}
\item[1)] segments $s^i$ and $s^j$ are distinct,
%\item[2)] \textcolor{red}{(implied by 3)?) segment $s^i$ and not a split of $s^j$ and $s^j$ is not a split of $s^i$, and}
\item[2)] neither $s^i$ nor $s^j$ is expressible as a composition of at least two segments in $J(s^i_{\mathcal{D}^i})\cup J(s^j_{\mathcal{D}^j})$, and
\item[3)] neither $J(s^i_{\mathcal{D}^i}) \le J(s^{j}_{\mathcal{D}^j})$ nor $J(s^{j}_{\mathcal{D}^j}) \le J(s^i_{\mathcal{D}^i})$.
\end{itemize}
for any distinct $i,j \in \{1,\ldots, k\}$.
\end{theorem}

\begin{figure}
\includegraphics[scale=1]{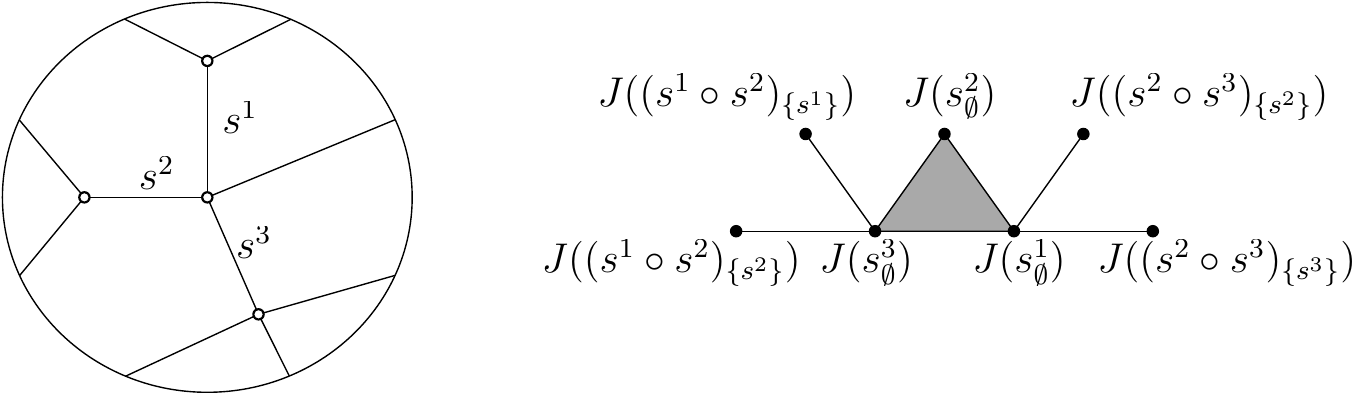}
\caption{A tree $T$ and the canonical join complex $\Delta^{CJ}(\text{Bic}(T))$. The shortest segments of $T$ are labeled $s^1, s^2,$ and $s^3$.}
\label{CJ_complex_figure}
\end{figure}

\begin{example}
Let $T$ be the tree shown in Figure~\ref{CJ_complex_figure}. In this same figure, we show the canonical join complex of the biclosed sets of $T$. Each join-irreducible of $\text{Bic}(T)$ is written next to its corresponding vertex of $\Delta^{CJ}(\text{Bic}(T))$.
\end{example}

\begin{proof}[Proof of Theorem~\ref{Thm_CJC}]
Let $\{J(s^1_{\mathcal{D}^1}), \ldots, J(s^k_{\mathcal{D}^k})\} \subset \text{JI}(\text{Bic}(T))$ where there exists distinct $i, j \in \{1,\ldots, k\}$ such that $s^i_{\mathcal{D}^i}$ and  $s^j_{\mathcal{D}^j}$ do not satisfy all of the stated properties. If $s^i = s^j$, then Lemma~\ref{nosplit} implies that there does not exist $B \in \text{Bic}(T)$ such that $s^i, s^j \in \lambda_{\downarrow}(B)$. Therefore, there does not exist $B \in \text{Bic}(T)$ such that $s^i_{\mathcal{D}}, s^j_{\mathcal{D}^\prime} \in \widetilde{\lambda}_\downarrow(B)$ for any subsets $\mathcal{D}, \mathcal{D}^\prime \subset \text{Seg}(T)$. Now by Lemma~\ref{Lemma_cjr_cmr}, we have that $J(s^i_{\mathcal{D}^i})\vee J(s^j_{\mathcal{D}^j})$ is not a canonical join representation.

Next, without loss of generality, suppose that $s^i$ may be expressed as a composition of at least two segments in $J(s^{i}_{\mathcal{D}^i}) \cup J(s^j_{\mathcal{D}^j})$. Then Lemma~\ref{lemma_comp_not_canonical} implies that $J(s^i_{\mathcal{D}^i})\vee J(s^j_{\mathcal{D}^j})$ is not a canonical join representation. By \cite[Theorem 1.1]{barnard2016canonical}, we obtain that $\{J(s^1_{\mathcal{D}^1}), \ldots, J(s^k_{\mathcal{D}^k})\}$ is not a face of $\Delta^{CJ}(\text{Bic}(T))$.

Lastly, suppose that, without loss of generality, $J(s^i_{\mathcal{D}^i}) \le J(s^{j}_{\mathcal{D}^j})$. This implies that $J(s^i_{\mathcal{D}^i}) \vee J(s^{j}_{\mathcal{D}^j}) = J(s^{j}_{\mathcal{D}^j})$ and so the expression $J(s^i_{\mathcal{D}^i}) \vee J(s^{j}_{\mathcal{D}^j})$ is not an irredundant join representation.

Conversely, suppose $\{J(s^1_{\mathcal{D}^1}), \ldots, J(s^k_{\mathcal{D}^k})\} \subset \text{JI}(\text{Bic}(T))$ and that any pair of distinct labels $s^i_{\mathcal{D}^i}$ and $s^j_{\mathcal{D}^j}$  satisfy all of the stated properties. Then Lemma~\ref{Lemma_is_canonical} implies that $J(s^i_{\mathcal{D}^i})\vee J(s^j_{\mathcal{D}^j})$ is a canonical join representation for any distinct $i,j \in \{1,\ldots, k\}$. Now, using \cite[Theorem 1.1]{barnard2016canonical}, we have that $\bigvee_{i = 1}^k J(s^i_{\mathcal{D}^i})$ is a canonical join representation. Thus $\{J(s^1_{\mathcal{D}^1}), \ldots, J(s^k_{\mathcal{D}^k})\} \subset \text{JI}(\text{Bic}(T))$ is a face of $\Delta^{CJ}(\text{Bic}(T))$.
\end{proof}

The rest of this section focuses on proving the lemmas used in the proof of Theorem~\ref{Thm_CJC}. The next two results will be used in several places in the remainder of the paper, whereas Lemmas~\ref{lemma_comp_not_canonical} and \ref{Lemma_is_canonical} are only used directly in the proof of Theorem~\ref{Thm_CJC}.

\begin{lemma}\label{Lemma_join_irr_containment}
Let $s_\mathcal{D} \in \mathcal{S}_T$, let $t \in J(s_\mathcal{D})$, and let $\mathcal{D}(t):= \{s^\prime \in J(s_\mathcal{D}): \text{$s^\prime$ is a split of $t$}\}$. Then $J(t_{\mathcal{D}(t)}) \le J(s_\mathcal{D})$. Moreover, $J(t_{\mathcal{D}(t)})$ consists of subsegments of $t$ that belong to $J(s_\mathcal{D})$.
\end{lemma}
\begin{proof}
By Lemma~\ref{Lemma_J_biclosed}, $J(s_\mathcal{D})$ is biclosed and no two elements of $J(s_\mathcal{D})$ are composable, we see that for each break $\{s_1,s_2\}$ of $t$ exactly one of these splits belongs to $J(s_\mathcal{D})$. This implies that $t_{\mathcal{D}(t)} \in \mathcal{S}_T$.

Now, we show that $J(t_{\mathcal{D}(t)}) \le J(s_\mathcal{D})$. Clearly, $\{t\} \sqcup \mathcal{D}(t) \subset J(s_\mathcal{D}).$ Now consider $S(s^\prime, \mathcal{D}(t))$ for some $s^\prime \in \mathcal{D}(t)$. Either $s^\prime \in \mathcal{D}$ or $s^\prime \in S(t, \mathcal{D})$ so it suffices to assume the latter. 

Let $t^\prime \in S(s^\prime, \mathcal{D}(t))$. Now write $s^\prime = t^\prime \circ s_3$, $t = s_1 \circ t^\prime \circ s_3$, and $s = s_4 \circ s_1\circ t^\prime \circ s_3\circ s_2$ for some $s_1, s_2, s_3, s_4 \in \text{Seg}(T)$. Since $s^\prime \in S(t, \mathcal{D})$, we know that $s_4\circ s_1$ and $s_2$, which are splits of $s$, do not belong to $\mathcal{D}$. This means that $s^\prime \circ s_2 \in \mathcal{D}$. Since $t^\prime \in S(s^\prime, \mathcal{D}(t))$, we also know that $s_1$ and $s_3$, which are splits of $t$, do not belong to $\mathcal{D}(t)$. This implies that $t^\prime \in S(s^\prime\circ s_2, \mathcal{D}) \subset J(s_\mathcal{D})$.

Lastly, let $t^\prime \in J(s_\mathcal{D})$ where $t^\prime \subseteq t.$ We can assume that $t^\prime$ is not a split of $t$. Now write $s = t^\prime_1\circ t^\prime \circ t_2^\prime$ for some $t_1^\prime, t_2^\prime \in \text{Seg}(T)$. Since $t^\prime \in J(s_\mathcal{D})$ it is not composable with any element of $J(s_\mathcal{D})$, we know that $t^\prime_1, t^\prime_2 \not \in \mathcal{D}$. This implies that $t^\prime_1\circ t^\prime, t^\prime \circ t^\prime_2 \in \mathcal{D}$ and $t^\prime \in S(t^\prime_1\circ t^\prime, \mathcal{D})\cap S(t^\prime \circ t^\prime_2, \mathcal{D})$. 

Now write $t^\prime_1 = t_1\circ s_1$ and $t^\prime_2 = s_2 \circ t_2$ where $s_1$ and $s_2$ are splits of $t$. As $t^\prime$ is not composable with any element of $J(s_\mathcal{D})$, we know that $s_1, s_2 \not \in J(s_\mathcal{D}).$ Thus $s_1\circ t^\prime, t^\prime \circ s_2 \in \mathcal{D}(t)$. It follows that $t^\prime \in S(s_1\circ t^\prime, \mathcal{D}(t))\cap S(t^\prime \circ s_2, \mathcal{D}(t)) \subset J(t_{\mathcal{D}(t)})$. This completes the proof.
\end{proof}

%Lastly, let $t^\prime \in J(s_\mathcal{D})$ where $t^\prime \subseteq t.$ We can assume that $t^\prime$ is not a split of $t$. Now write $s = t_1\circ t^\prime \circ t_2$ and $t = t_1\circ t^\prime \circ t_2^\prime$ for some $t_1, t_2, t_2^\prime \in \text{Seg}(T)$ where $t_2^\prime \subseteq t$. Since $t^\prime \in J(s_\mathcal{D})$ it is not composable with any element of $J(s_\mathcal{D})$. Thus we know that $t_1, t_2 \not \in \mathcal{D}$ and $t_2^\prime \not \in J(s_\mathcal{D})$. This implies that $t_1\circ t^\prime, t^\prime \circ t_2 \in \mathcal{D}$ and $t^\prime \in S(t_1\circ t^\prime, \mathcal{D})\cap S(t^\prime \circ t_2, \mathcal{D})$. Since $t_2^\prime \not \in J(s_\mathcal{D})$, we also know that $t^\prime \in S(s_1, \mathcal{D}(t))$. This completes the proof.

\begin{lemma}\label{nosplit}
Given a biclosed set $B \in \text{Bic}(T)$ and distinct covering relations $(B_1,B), (B_2,B) \in \text{Cov}(\text{Bic}(T))$, the segment $s_1 = \lambda(B_1,B)$ is not a split of $s_2 = \lambda(B_2,B)$ and $s_1 \neq s_2$.
\end{lemma}

\begin{proof}
Since $B_1 \neq B_2$, it is clear that $s_1\neq s_2$.

Next, suppose for the sake of contradiction $s_1 = \lambda(B_1,B)$ is a split of $s_2 = \lambda(B_2,B)$. Then $B_1=B\backslash\{s_1\}$ and $B_2=B\backslash\{s_2\}$. Now let $t \in \text{Seg}(T)$ denote the segment satisfying $s_1 \circ t = s_2$. Observe that since $s_1 \not \in B_1$, we have $t \in B_1$. This implies that $t \in B$ and so $t \in B_2$. However, this means that $s_1, t \in B_2$, but $s_2 = s_1\circ t \not \in B_2$, which contradicts that $B_2$ is closed.
\end{proof}

\begin{lemma}\label{lemma_comp_not_canonical}
Given $s_\mathcal{D}, s^\prime_{\mathcal{D}^\prime} \in \mathcal{S}_T$. Assume that there exists $t^1, \ldots, t^k \in J(s_\mathcal{D})\cup J(s^\prime_{\mathcal{D}^\prime})$ with $k \ge 2$ such that $s = t^1\circ \cdots \circ t^k$. Then $J({s}_\mathcal{D}) \vee J(s^\prime_{\mathcal{D}^\prime})$ is not a canonical join representation.
\end{lemma}
\begin{proof}
We can assume that $s \neq s^\prime$, $s^\prime$ is not a split of $s$, and $s$ is not a split of $s^\prime$, otherwise we obtain the desired result from Lemma~\ref{nosplit} and Lemma~\ref{Lemma_cjr_cmr}. 

Assume $s = t^1\circ \cdots \circ t^k$ with $k \ge 2$ for some $t^1, \ldots, t^k \in J(s_\mathcal{D})\cup J(s^\prime_{\mathcal{D}^\prime})$. We assume that segment $t^1$ as long as possible with the property that $t^1$ is not expressible as a composition of at least two elements of $J(s_\mathcal{D})\cup J(s^\prime_{\mathcal{D}^\prime})$. Inductively, by choosing $t^1, \ldots, t^{i}$ in this way, we can then choose $t^{i+1}$ as long as possible with the property that $t^{i+1}$ is not expressible as a composition of at least two elements of $J(s_\mathcal{D})\cup J(s^\prime_{\mathcal{D}^\prime})$. Now since no two elements of $J(s_\mathcal{D})$ are composable and no two elements of $J(s^\prime_{\mathcal{D}^\prime})$ are composable, without loss of generality, we have that $t^1, t^3, \ldots \in J(s_\mathcal{D})$ and $t^2, t^4, \ldots \in J(s^\prime_{\mathcal{D}^\prime})$.

%Assume $s = t^1\circ \cdots \circ t^k$ with $k \ge 2$ for some $t^1, \ldots, t^k \in J(s_\mathcal{D})\cup J(s^\prime_{\mathcal{D}^\prime})$. We assume that segment $t^1$ as long as possible with the property that $t^1$ is not expressible as a composition of at least two elements of $J(s_\mathcal{D})\cup J(s^\prime_{\mathcal{D}^\prime})$. Inductively, by choosing $t^1, \ldots, t^{i}$ in this way, we can then choose $t^{i+1}$ is as long as possible with the property that $t^{i+1}$ is not expressible as a composition of at least two elements of $J(s_\mathcal{D})\cup J(s^\prime_{\mathcal{D}^\prime})$. We can assume that $s^\prime$ is not a split of $s$, otherwise we obtain the desired result from Lemma~\ref{nosplit}. Without loss of generality, segment $t^1 \not \subseteq s^\prime.$ Now since no two elements of $J(s_\mathcal{D})$ are composable and no two elements of $J(s^\prime_{\mathcal{D}^\prime})$ are composable, we have that $t^1, t^3, \ldots \in J(s_\mathcal{D})$ and $t^2, t^4, \ldots \in J(s^\prime_{\mathcal{D}^\prime})$.

To complete the proof, it is enough to show that $$J(t^1_{\mathcal{D}(t^1)})\vee J(t^3_{\mathcal{D}(t^3)})\vee \cdots \vee J(s^\prime_{\mathcal{D}^\prime}) = J(s_\mathcal{D})\vee J(s^\prime_{\mathcal{D}^\prime})$$ as the former is a refinement of the latter. By Lemma~\ref{Lemma_join_irr_containment}, we see that the set on the left-hand side is contained in $J(s_\mathcal{D})\vee J(s^\prime_{\mathcal{D}^\prime})$. Moreover, if $t \in J(s_\mathcal{D})\vee J(s^\prime_{\mathcal{D}^\prime})$ satisfies $t \subseteq t^i$ for some $i = 1, \ldots, k$, then $t$ belongs to $J(t^1_{\mathcal{D}(t^1)})\vee J(t^3_{\mathcal{D}(t^3)})\vee \cdots \vee J(s^\prime_{\mathcal{D}^\prime}).$ This means that to prove the opposite containment, we must show that $t \in J(t^1_{\mathcal{D}(t^1)})\vee J(t^3_{\mathcal{D}(t^3)})\vee \cdots \vee J(s^\prime_{\mathcal{D}^\prime})$ when one of the following holds:
\begin{itemize}
\item[\textit{1)}] $t = t^1 \circ t^2 \circ \cdots \circ t^{j-1}\circ t_j^\prime$ for some $t_j^\prime \in J(s_\mathcal{D})\cup J(s^\prime_{\mathcal{D}^\prime})$,%$t_j^\prime \in J(t^j_{\mathcal{D}(t^j)})$,
\item[\textit{2)}] $t = t_i^\prime \circ t^{i+1} \circ \cdots \circ t^{k-1}\circ t^k$ for some $t_i^\prime \in J(s_\mathcal{D})\cup J(s^\prime_{\mathcal{D}^\prime})$, or %\in J(t^i_{\mathcal{D}(t^i)})$, or
\item[\textit{3)}] $t = t_i^\prime \circ t^{i+1} \circ \cdots \circ t^{j-1}\circ t_j^\prime$ for some $t_j^\prime, t_i^\prime \in J(s_\mathcal{D})\cup J(s^\prime_{\mathcal{D}^\prime})$.%. $t_j^\prime \in J(t^j_{\mathcal{D}(t^j)})$ and $t_i^\prime \in J(t^i_{\mathcal{D}(t^i)})$.
\end{itemize}
We verify \textit{Case 2)}, and the proof of \textit{Case 1)} and \textit{3)} is similar to that of \textit{Case 2)}.

\textit{Case 2):} We show that $t_i^\prime \in J(t^i_{\mathcal{D}(t^i)}).$ Suppose $t_i^\prime \not \in J(t^i_{\mathcal{D}(t^i)})$. Since $t_i^\prime$ is a split of $t^i$, we may write $t_i^{\prime\prime}\circ t_i^\prime = t^i$ for some $t^{\prime\prime}_i \in \text{Seg}(T)$. As $t_i^\prime$ does not belong to $J(t^i_{\mathcal{D}(t^i)})$, we know that $t_i^{\prime\prime} \in \mathcal{D}(t^i)$. We also know $t_i^\prime \in J(s_\mathcal{D})\cup J(s^\prime_{\mathcal{D}^\prime})$ and so the expression $t^i = t_i^{\prime\prime}\circ t_i^\prime$ contradicts our choice of $t^i$.
\end{proof}

%\textit{Case 2):} We show that $t_i^\prime \in J(t^i_{\mathcal{D}(t^i)}).$ Suppose $t_i^\prime \not \in J(t^i_{\mathcal{D}(t^i)})$. Since $t_i^\prime$ is a split of $t^i$, we may write $t_i^{\prime\prime}\circ t_i^\prime = t^i$ for some $\text{Seg}(T)$. As $t_i^\prime$ does not belong to $J(t^i_{\mathcal{D}(t^i)})$, we know that $t_i^{\prime\prime} \in \mathcal{D}(t^i)$. We also know $t_i^\prime \in J(s_\mathcal{D})\cup J(s^\prime_{\mathcal{D}^\prime})$ and so the expression $t^i = t_i^{\prime\prime}\circ t_i^\prime$ contradicts our choice of $t^\prime$.

\begin{lemma}\label{Lemma_is_canonical}
Let ${s}_\mathcal{D}, s^\prime_{\mathcal{D}^\prime} \in \mathcal{S}_T$ be labels with the following properties:
\begin{itemize}
\item[1)] segments $s$ and $s^\prime$ are distinct,
%\item[2)] \textcolor{red}{(implied by 3)?) segment $s$ and not a split of $s^\prime$ and $s^\prime$ is not a split of $s$, and}
\item[2)] neither $s$ nor $s^\prime$ is expressible as a composition of at least two segments in $J(s_\mathcal{D})\cup J(s^\prime_{\mathcal{D}^\prime})$, and
\item[3)] neither $J(s^i_{\mathcal{D}^i}) \le J(s^{j}_{\mathcal{D}^j})$ nor $J(s^{j}_{\mathcal{D}^j}) \le J(s^i_{\mathcal{D}^i})$.
\end{itemize}Then $J({s}_\mathcal{D}) \vee J(s^\prime_{\mathcal{D}^\prime})$ is a canonical join representation.
\end{lemma}
\begin{proof}
By the stated properties satisfied by $s_\mathcal{D}$ and $s^\prime_{\mathcal{D}^\prime},$ we have that there exist segments $t \in J(s_{\mathcal{D}})\backslash J(s^\prime_{\mathcal{D}^\prime})$ and $t^\prime \in J(s^\prime_{\mathcal{D}^\prime})\backslash J(s_{\mathcal{D}})$. This implies that $J(s_\mathcal{D}) < J(s_\mathcal{D})\vee J(s^\prime_{\mathcal{D}^\prime})$ and  $J(s^\prime_{\mathcal{D}^\prime}) < J(s_\mathcal{D})\vee J(s^\prime_{\mathcal{D}^\prime}).$ Therefore, the join representation $J(s_\mathcal{D})\vee J(s^\prime_{\mathcal{D}^\prime})$ is irredundant.

%\textcolor{red}{As $s$ and $s^\prime$ are distinct and neither $s$ nor $s^\prime$ is expressible as a composition of at least two segments from $J(s_\mathcal{D})\cup J(s^\prime_{\mathcal{D}^\prime})$, we know that the join representation $J(s_\mathcal{D})\vee J(s^\prime_{\mathcal{D}^\prime})$ is irredundant.}

Next, suppose that $J(s_\mathcal{D})\vee J(s^\prime_{\mathcal{D}^\prime}) = \bigvee_{i = 1}^k J(t^i_{\mathcal{D}^i})$ where the latter is irredundant. We will show that $J(s_\mathcal{D}) \le J(t^i_{\mathcal{D}^i})$ for some $i = 1, \ldots, k$, and one uses the same strategy to prove that $J(s^\prime_{\mathcal{D}^\prime}) \le J(t^j_{\mathcal{D}^j})$ for some $j = 1, \ldots, k$.

Since $s \in \bigvee_{i = 1}^k J(t^i_{\mathcal{D}^i})$, there exist $t_{i_j} \in J(t^{i_j}_{\mathcal{D}^{i_j}})$ with $j = 1, \ldots, \ell$ such that $s = t_{i_1}\circ \cdots \circ t_{i_\ell}$. By possibly factoring the segments $t_{i_j}$ further and by the fact that $J(s_\mathcal{D})\vee J(s^\prime_{\mathcal{D}^\prime}) = \bigvee_{i = 1}^k J(t^i_{\mathcal{D}^i})$, we can assume $t_{i_j} \in J(s_\mathcal{D})\cup J(s^\prime_{\mathcal{D}^\prime})$ for all $j = 1, \ldots, \ell$. As $s$ is not expressible as a composition of at least two segments from $J(s_\mathcal{D})\cup J(s^\prime_{\mathcal{D}^\prime})$, this implies that $\ell = 1$ and so $s \in J(t^i_{\mathcal{D}^i})$ for some $i = 1, \ldots, k$.

Now let $t \in \mathcal{D}$. We can write $s = t \circ t^\prime$ for some $t^\prime \in \text{Seg}(T)$. Suppose $t \not \in J(t^i_{\mathcal{D}^i})$.  Since $J(t^i_{\mathcal{D}^i})$ is biclosed and $s \in J(t^i_{\mathcal{D}^i})$, we know $t^\prime \in J(t^i_{\mathcal{D}^i}).$ However, by the fact that $J(s_\mathcal{D})\vee J(s^\prime_{\mathcal{D}^\prime}) = \bigvee_{i = 1}^k J(t^i_{\mathcal{D}^i})$, the equation $s = t \circ t^\prime$ contradicts that $s$ is not expressible as a composition of at least two segments from $J(s_\mathcal{D})\cup J(s^\prime_{\mathcal{D}^\prime})$. Thus $t \in J(t^i_{\mathcal{D}^i})$ so $\mathcal{D} = \{t \in J(t^i_{\mathcal{D}^i}): \ t \text{ is a split of } s\}.$

We now conclude from Lemma~\ref{Lemma_join_irr_containment} that $J(s_\mathcal{D}) \le J(t^i_{\mathcal{D}^i})$ so $J({s}_\mathcal{D}) \vee J(s^\prime_{\mathcal{D}^\prime})$ is a canonical join representation.\end{proof}

\section{Shard intersection order}\label{shard_section}

In this section, we describe elements of $\Psi(\text{Bic}(T))$ using our CU-labeling of $\text{Bic}(T)$. After that, we use this description to show that $\Psi(\text{Bic}(T))$ is a lattice.

\subsection{Elements of the shard intersection order} Let $B \in \text{Bic}(T)$ be a biclosed set that covers exactly the following biclosed sets: $B_1, B_2,\ldots, B_k$. Let ${s_i}_{\mathcal{D}_i} = \widetilde{\lambda}(B_i,B)$ for $i=1, \ldots, k$ and $\lambda_\downarrow(B) =\{s_1, s_2,\ldots,s_k\}$ where $s_i = \lambda(B_i,B)$ for $i = 1, \ldots, k$. Now fix a segment $s \in\overline{\lambda_{\downarrow}(B)}$ expressed as $s = s_{i_1}\circ s_{i_2}\circ \ldots\circ s_{i_\ell}$ with each $s_{i_j}\in \lambda_{\downarrow}(B)$. If $t \in \text{Seg}(T)$ is a split of $s$ that can be expressed as either $t = s_{i_1}\circ \cdots \circ s_{i_j}$ for some $j = 1, \ldots, \ell-1$ or $t = s_{i_j}\circ \cdots \circ s_{i_\ell}$ for some $j = 2,\ldots, \ell$, we say that $t$ is a \textbf{faultline split} of $s$. Otherwise, we say that $t$ is a \textbf{non-faultline split} of $t$. Additionally, we refer to $\{s_{i_1}\circ\cdots \circ s_{i_j}, s_{i_{j+1}}\circ \cdots \circ s_{i_\ell}\}$ as a \textbf{faultine break}.

%(\textcolor{red}{or  $\widetilde{\lambda}_i$ ${s_i}_{\delta_i} := \widetilde{\lambda}(B_i,B)$})

\begin{theorem}\label{psib}
	Given a biclosed set $B \in \text{Bic}(T)$, we have that $\psi(B)$ is the set of all labels of the form $$(s_{i_1}\circ s_{i_2}\circ \cdots\circ s_{i_\ell})_{\mathcal{D}}$$ with $s_{i_j} \in \lambda_{\downarrow}(B)$ where $s_{i_1}\circ s_{i_2}\circ \cdots\circ s_{i_\ell}$ is any element of $\overline{\lambda_{\downarrow}(B)}$ and where $\mathcal{D}$ is any set of segments that satisfies the following properties: 
\begin{itemize}
\item[(i)] $|\mathcal{D}| = |\{\text{breaks of } s_{i_1}\circ s_{i_2}\circ \cdots\circ s_{i_\ell}\}|,$
\item[(ii)] each segment $t \in \mathcal{D}$ is a split of $s_{i_1}\circ \cdots \circ s_{i_\ell}$,
\item[(iii)] no two distinct splits $t_1, t_2 \in \mathcal{D}$ appear in the same break of $s_{i_1}\circ \cdots \circ s_{i_\ell}$, and
\item[(iv)] {whenever $t \in \mathcal{D}$ is a non-faultline split of $s_{i_1}\circ \cdots \circ s_{i_\ell}$, we have that $t = s_{i_1} \circ \cdots \circ s_{i_{j-1}}\circ t_j$ for some $j = 1,\ldots, \ell$ and some $t_j \in \mathcal{D}_{i_j}$ or $t = t_j \circ s_{i_{j+1}}\circ \cdots \circ s_{i_{\ell}}$ for some $j = 1,\ldots, \ell$ and some $t_j \in \mathcal{D}_{i_j}$. In the former case if $j = 1$, we mean $t = t_1$, and in the latter case, if $j = \ell$, we mean $t = t_\ell$.}
\end{itemize}\end{theorem}

%\begin{remark}\label{non-faultline_remark}
%We obtain the following consequence of Theorem~\ref{psib}, which is stated using the notation of the theorem. If $s_{i_1}\circ \cdots \circ s_{i_\ell} \in B$ and segments $t$ and $t^\prime$ are non-faultline splits of $s_{i_1}\circ \cdots \circ s_{i_\ell}$ where $s_{i_1}\circ \cdots \circ s_{i_\ell} = t\circ t^\prime,$ then exactly one of $t$ and $t^\prime$ belongs to $B.$
%\end{remark}

\begin{figure}
\centering
\includegraphics[scale=1.2]{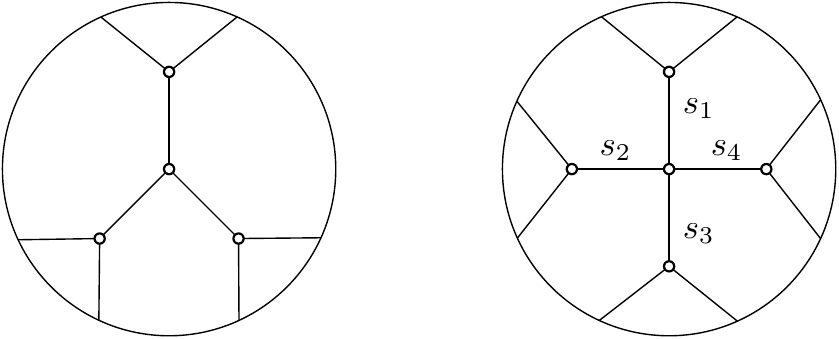}
\caption{The trees from Example~\ref{shard_calc_ex}.}
\label{cycle_trees}
\end{figure}

\begin{example}\label{shard_calc_ex}
Let $T$ be the tree on the left in Figure~\ref{cycle_trees}, and let $T^\prime$ be the tree on the right. In Figure~\ref{shard_ex_a3}, we show the shard intersection order of $\text{Bic}(T)$. The atoms in this lattice are the 9 labels in $\mathcal{S}_T$. The presence of a purple segment $s$ indicates that both labels $s_\mathcal{D}$ and $s_{\mathcal{D}^\prime}$ belong to the corresponding set $\psi(B)$. This indicates that given $B \in \text{Bic}(T)$, one has $$|\psi(B)| = |\{\text{dark red segments in $\psi(B)$}\}| + 2|\{\text{purple segments in $\psi(B)$}\}|.$$

The shard intersection order of biclosed sets fails to be graded in general, although $\Psi(\text{Bic}(T))$ is graded of rank 3. We show that $\Psi(\text{Bic}(T^\prime))$ is not graded by producing two maximal chains of $\Psi(\text{Bic}(T^\prime))$ with different lengths. The first maximal chain is as follows:
$$\begin{array}{l}
\emptyset\\ \{(s_1\circ s_2)_{\{s_1\}}\}\\ \{(s_1\circ s_2)_{\{s_1\}}, (s_1\circ s_4)_{\{s_1\}}\}\\ \{(s_1\circ s_2)_{\{s_1\}} (s_1\circ s_4)_{\{s_1\}}, (s_3\circ s_4)_{\{s_3\}}\}\\ \{(s_1\circ s_2)_{\{s_1\}}, (s_1\circ s_4)_{\{s_1\}}, (s_3\circ s_4)_{\{s_3\}}, (s_2\circ s_3)_{\{s_3\}}\}\\ \mathcal{S}_{T^\prime}.
\end{array}$$

The second maximal chain is as follows:

$$\begin{array}{l}
\emptyset\\ \{(s_1\circ s_2)_{\{s_1\}}\}\\ \{(s_1\circ s_2)_{\{s_1\}}, {s_4}_{\emptyset}\}\\ \{(s_1\circ s_2)_{\{s_1\}}, {s_4}_{\emptyset}, {s_2}_{\emptyset}, {s_1}_{\emptyset}, (s_1\circ s_2)_{\{s_2\}}, (s_1\circ s_4)_{\{s_1\}}, (s_1\circ s_4)_{\{s_4\}}\}\\ \mathcal{S}_{T^\prime}. %\{s(s_1s_2)_{\{s_1\}} (s_1s_4)_{\{s_1\}}, (s_3s_4)_{\{s_3\}}\}\\ \{s(s_1s_2)_{\{s_1\}}, (s_1s_4)_{\{s_1\}}, (s_3s_4)_{\{s_3\}}, (s_2s_3)_{\{s_3\}}\}\\ \mathcal{S}_{T^\prime}.
\end{array}$$

That $\Psi(\text{Bic}(T))$ generally fails to be graded was already observed in \cite[Remark 6.12]{cliftondillery}. The trees $T$ and $T^\prime$ in this example belong to the one parameter family of trees that are completely determined by the choice of degree on their central vertex. In \cite[Conjecture 6.13]{cliftondillery}, the first and second authors conjectured that for this one parameter family of trees $\Psi(\text{Bic}(T))$ is graded if and only if $n$ is odd.

\end{example}

\begin{figure}
\centering
\rotatebox{270}{\includegraphics[scale=1.45]{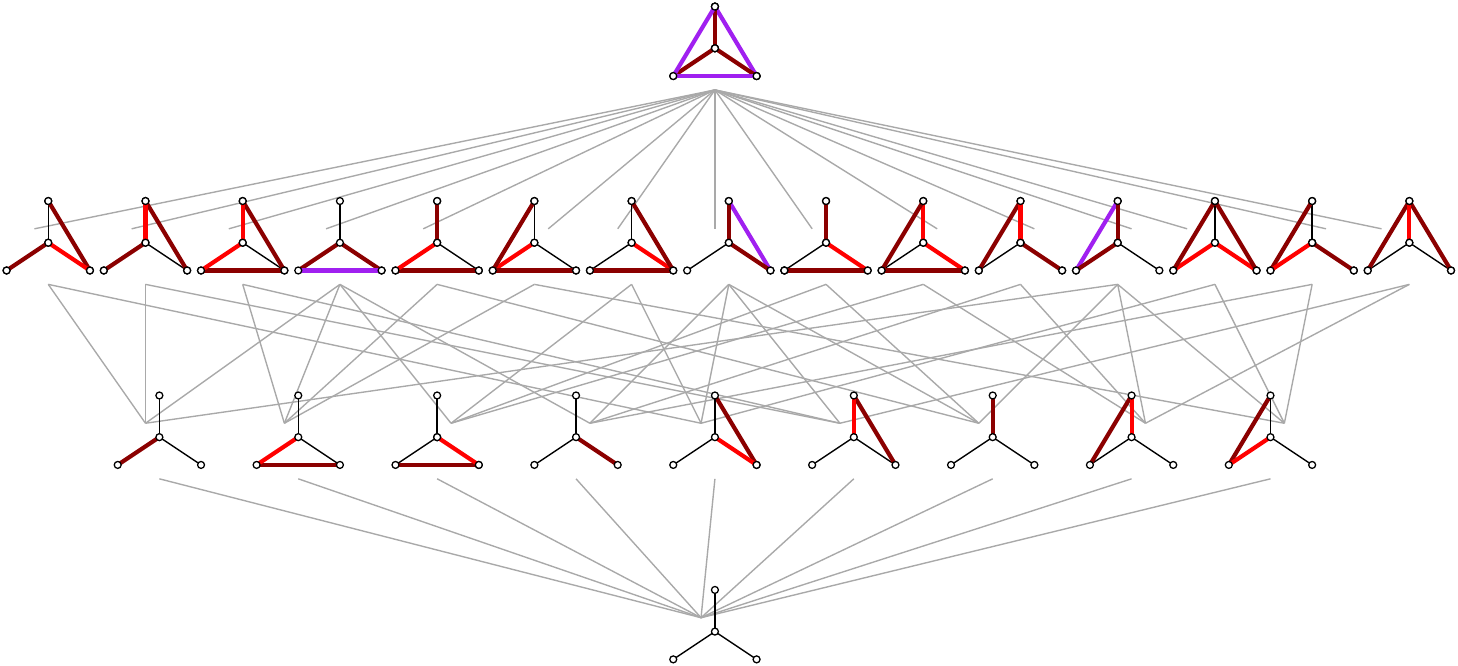}}
\caption{The shard intersection order $\Psi(\text{Bic}(T))$ when $T$ is the tree on the left in Figure~\ref{cycle_trees}.}
\label{shard_ex_a3}
\end{figure}

We use the following lemma in the proof of Theorem~\ref{psib}.

%Given a biclosed $B \in \text{Bic}(T)$ whose lower covers are $B_1, \ldots, B_k \in \text{Bic}(T)$. We have that $\bigwedge_{i = 1}^k B_i\sqcup \{s_j\}$ is a biclosed set where $s_j = \lambda(B_j, B)$ for any $j = 1, \ldots, k.$

\begin{lemma}\label{Lemma_facial_int_atoms}
Given a biclosed $B \in \text{Bic}(T)$ where $B_1, \ldots, B_k \in \text{Bic}(T)$ are all of the biclosed sets covered by $B$ and where $s_j = \lambda(B_j, B)$ for $j = 1, \ldots, k$, we have that $B^\prime \in \left[\bigwedge_{i = 1}^k B_i, B\right]$ if and only if $B^\prime = (\bigwedge_{i = 1}^k B_i) \sqcup S^\prime$ where $S^\prime$ and $\overline{\{s_1,\ldots, s_k\}}\backslash S^\prime$ are closed subsets of $\overline{\{s_1,\ldots, s_k\}}$. 
\end{lemma}
\begin{proof}
By \cite[Proposition 4.2]{garver2016oriented}, it is enough to show that $(\bigwedge_{i = 1}^k B_i)\sqcup \{s_j\}$ is a biclosed set for any $j = 1, \ldots, k.$ Note that $B = (\bigwedge_{i = 1}^k B_i) \sqcup \overline{\{s_1, \ldots, s_k\}}$ by Lemma~\ref{meet_of_B_is}.

To show that $(\bigwedge_{i = 1}^k B_i)\sqcup \{s_j\}$ is closed, it is enough to show that if $s \in \bigwedge_{i = 1}^k B_i$ and segments $s$ and $s_j$ are composable, then $s \circ s_j \in (\bigwedge_{i = 1}^k B_i) \sqcup \{s_j\}$. Since $B$ is closed and $B = (\bigwedge_{i = 1}^k B_i) \sqcup \overline{\{s_1, \ldots, s_k\}}$, we have that $s \circ s_j \in B$. Therefore, we can assume that $s \circ s_j \not \in \bigwedge_{i = 1}^k B_i$. Write $s\circ s_j = s_{i_1}\circ \cdots \circ s_{i_\ell}$ where $s_{i_1}, \ldots, s_{i_\ell} \in \{s_1, \ldots, s_k\}$, and we can assume that $\ell \ge 2$. By Lemma~\ref{nosplit}, we see that $s_j = s_{i_1}$ or $s_j = s_{i_\ell}$. We conclude that $s \in \overline{\{s_1, \ldots, s_k\}}$, a contradiction.

%\textcolor{red}{FIX AFTER THIS}

Next, we show that $((\bigwedge_{i = 1}^k B_i)\sqcup \{s_j\})^c$ is closed. Note that $((\bigwedge_{i = 1}^k B_i)\sqcup \{s_j\})^c = (\bigvee_{i = 1}^k B_i^c)\backslash\{s_j\}$ by Lemma~\ref{Lemma_meet_in_bic}. It is clear that if $s, t \in (\bigvee_{i = 1}^k B_i^c)\backslash\{s_j\}$ are composable, then $s\circ t \in (\bigvee_{i = 1}^k B_i^c)$. Thus, it remains to show that $s\circ t \neq s_j$. 

Suppose $s\circ t = s_j$. If $s, t \in  (\bigvee_{i = 1}^k B_i^c)\backslash\overline{\{s_1,\ldots, s_k\}}$, then since $(\bigvee_{i = 1}^k B_i^c)\backslash\overline{\{s_1,\ldots, s_k\}}$ is a closed set, we know that $s\circ t \in (\bigvee_{i = 1}^k B_i^c)\backslash\overline{\{s_1,\ldots, s_k\}}$. However, this contradicts that $s\circ t = s_j$. We conclude that $s \in \overline{\{s_1,\ldots, s_k\}}$ or $t \in \overline{\{s_1,\ldots, s_k\}}$. We assume $s \in \overline{\{s_1,\ldots, s_k\}}$, without loss of generality. Write $s = s_{i_1}\circ \cdots \circ s_{i_\ell}$ for some $s_{i_1}, \ldots, s_{i_\ell} \in \{s_1,\ldots, s_k\}$. Since $s\circ t = s_j$, we have $s_{i_1}\circ \cdots \circ s_{i_\ell} \circ t = s_j$. This implies that $s_{i_1}$ is a split of $s_j$, and this contradicts Lemma~\ref{nosplit}.
\end{proof}

%By Lemma~\ref{Lemma_facial_int_atoms} and \cite[Proposition 4.2]{garver2016oriented}, we obtain that $B^\prime \in \left[\bigwedge_{i = 1}^k B_i, B\right]$ if and only if $B^\prime = (\bigwedge_{i = 1}^k B_i) \sqcup S^\prime$ where $S^\prime$ and $\overline{\{s_1,\ldots, s_k\}}\backslash S^\prime$ are closed subsets of $\overline{\{s_1,\ldots, s_k\}}$.

\begin{proof}[Proof of Theorem~\ref{psib}]
%Let $B$ be any biclosed set whose lower covers are $B_1,\ldots, B_k \in \text{Bic}(T)$ and with $s_i := \lambda(B_i,B)$ for $i = 1, \ldots, k$. 

We know that the set $\psi(B)$ consists of labels of the form $(s_{i_1}\circ \cdots \circ s_{i_\ell})_{\mathcal{D}}$ where $\mathcal{D}$ is some subset of $\text{Seg}(T)$ and $s_{i_1},\ldots, s_{i_\ell} \in \lambda_\downarrow(B)$. Since any biclosed set $B^\prime \in \left[\bigwedge_{i=1}^k B_i, B\right]$ is of the form $\bigwedge_{i=1}^k B_i \sqcup S^\prime$ where $S^\prime$ is some biclosed subset of $\overline{\{s_1,\ldots, s_k\}}$, the segment $s_{i_1}\circ \cdots \circ s_{i_\ell}$ may be any element of $\overline{\{s_1,\ldots, s_k\}}$. Moreover, we know that any two sets $\mathcal{D}$ and $\mathcal{D}^\prime$ appearing on labels $(s_{i_1}\circ \cdots \circ s_{i_\ell})_{\mathcal{D}}$ and $(s_{i_1}\circ \cdots \circ s_{i_\ell})_{\mathcal{D}^\prime}$ contain the same non-faultline splits. %We denote the common set of non-faultline splits by $\mathcal{N}$.

It remains to show that given a segment $s_{i_1}\circ \cdots \circ s_{i_\ell}$, the subsets $\mathcal{D}$ appearing on labels $(s_{i_1}\circ \cdots \circ s_{i_\ell})_{\mathcal{D}}$ are exactly those sets of splits of $s_{i_1}\circ \cdots \circ s_{i_\ell}$ with the properties appearing in the statement of the theorem. 

First, given a label $(s_{i_1}\circ \cdots \circ s_{i_\ell})_{\mathcal{D}} \in \psi(B)$, we show that $\mathcal{D}$ has the desired properties. It is clear from the definition of $\widetilde{\lambda}$ that $\mathcal{D}$ satisfies properties (i), (ii), and (iii). Thus, we proceed by induction on $\ell$ to prove property (iv). Assume that $\ell = 1$ so we consider labels of the form $(s_{i_1})_\mathcal{D}$. Observe that there are no nontrivial faultline splits of $s_{i_1}$ so $\mathcal{D}$ contains no faultline splits. If $({s_{i_1}})_{\mathcal{D}_{i_1}}$ is another label appearing on a covering relation in $\left[\bigwedge_{i=1}^k B_i, B\right]$, we have $|\mathcal{D}| = |\mathcal{D}_{i_1}|$. Thus $\mathcal{D} = \mathcal{D}_{i_1}$.

%We proceed by induction on $\ell$. First, we have $\ell = 1$ so we consider labels of the form $(s_{i_1})_\mathcal{D}$. Observe that there are no nontrivial faultline splits of $s_{i_1}$ so $\mathcal{D}$ contains no faultline splits. Now, as any biclosed set $B^\prime \in \left[\bigwedge_{i=1}^k B_i, B\right]$ is of the form $\bigwedge_{i=1}^k B_i \cup S^\prime$ where $S^\prime$ is some biclosed subset of $\overline{\{s_1,\ldots, s_k\}}$, we know that $\mathcal{D} \subset \mathcal{D}_{i_1}$. Since $|\mathcal{D}| = |\mathcal{D}_{i_1}|$, we know that $\mathcal{D} = \mathcal{D}_{i_1}$.

Next, suppose that for any label $(s_{j_1}\circ \cdots \circ s_{j_{\ell^\prime}})_{\mathcal{D}^\prime} \in \psi(B)$ where $s_{j_1},  \ldots, s_{j_{\ell^\prime}} \in \{s_{i_1}, \ldots, s_{i_\ell}\}$ and $\ell^\prime < \ell$ the subset $\mathcal{D}^\prime$ satisfies property (iv). Now consider $(s_{i_1}\circ \cdots \circ s_{i_\ell})_\mathcal{D} \in \psi(B)$ and let $t \in \mathcal{D}$ be a non-faultline split of $s_{i_1}\circ \cdots \circ s_{i_\ell}$. Without loss of generality, $t = s_{i_1}\circ \cdots \circ s_{i_{j-1}}\circ t_j$. Since $t_j$ is a non-faultline split of $s_{i_j}$, by induction, $t_j \in \mathcal{D}_{i_j}$. 

%To prove the converse, we show that given any label $(s_{i_1}\circ \cdots \circ s_{i_\ell})_\mathcal{D}$ where $s_{i_1}, \ldots, s_{i_\ell} \in \overline{\{s_1, \ldots, s_k\}}$ and $\mathcal{D}$ satisfies the properties in the statement of the proposition there exists a biclosed set $B^1 \in \left[\bigwedge_{i=1}^k B_i, B\right]$ such that $\widetilde{\lambda}(B^1, B^1 \sqcup \{s_{i_1}\circ \cdots \circ s_{i_\ell}\}) = (s_{i_1}\circ \cdots \circ s_{i_\ell})_{\mathcal{D}}.$ Define the biclosed set $$B^1 := \left(\bigwedge_{i=1}^k B_i\right) \vee (J((s_{i_1}\circ \cdots \circ s_{i_\ell})_\mathcal{D}))\backslash \{s_{i_1}\circ \cdots \circ s_{i_\ell}\}).$$ It is clear that $\bigwedge_{i=1}^k B_i < B^1$.

To prove the converse, it is enough to show that any family $\mathcal{F} \subset \overline{\{s_{i_{1}}, \ldots, s_{i_\ell}\}}$ of $\ell-1$ faultline splits of $s_{i_1}\circ \cdots \circ s_{i_\ell}$ satisfying the properties in the statement of the theorem belongs to a set $\mathcal{D}$ that appears on some label $(s_{i_1}\circ \cdots \circ s_{i_\ell})_\mathcal{D}$. Given $f = s_{i_1}\circ \cdots \circ s_{i_j} \in \mathcal{F}$, define $$D(f) := \{f =s_{i_1}\circ \cdots \circ s_{i_j}, s_{i_2}\circ \cdots \circ s_{i_j}, \ldots, s_{i_{j-1}} \circ s_{i_j}, s_{i_j} \}.$$ The set $D(f)$ is defined analogously when $f = s_{i_j}\circ \cdots \circ s_{i_\ell}$. By construction, $D(f)$ and $\overline{\{s_{i_{1}}, \ldots, s_{i_\ell}\}}\backslash D(f)$ are closed subsets of $\overline{\{s_{i_{1}}, \ldots, s_{i_\ell}\}}$. 

Now observe that $\overline{\bigcup_{f \in \mathcal{F}} D(f)}$ and $\overline{\{s_{i_{1}}, \ldots, s_{i_\ell}\}}\backslash \overline{\bigcup_{f \in \mathcal{F}} D(f)}$ are closed subsets of $\overline{\{s_{i_{1}}, \ldots, s_{i_\ell}\}}$. This implies that $\overline{\bigcup_{f \in \mathcal{F}} D(f)} \sqcup \bigwedge_{i=1}^k B_i \in \left[\bigwedge_{i=1}^k B_i, B\right]$ is a biclosed set. Moreover, the only (necessarily) faultline splits of $s_{i_1}\circ \cdots \circ s_{i_\ell}$ in $\overline{\bigcup_{f \in \mathcal{F}} D(f)}$ are the elements of $\mathcal{F}$. Thus, setting $B^{1} := \overline{\bigcup_{f \in \mathcal{F}} D(f)} \sqcup \bigwedge_{i=1}^k B_i$ and $B^{2} := B^{1} \sqcup \{s_{i_1}\circ \cdots \circ s_{i_\ell}\}$, we obtain that $\widetilde{\lambda}(B^{1}, B^{2}) = (s_{i_1}\circ \cdots \circ s_{i_\ell})_{\mathcal{F} \sqcup \mathcal{N}}$ where $\mathcal{N}$ is the set of non-faultline splits of $s_{i_1}\circ \cdots \circ s_{i_\ell}$ that appears in every label of the form $(s_{i_1}\circ \cdots \circ s_{i_\ell})_\mathcal{D}$.\end{proof}

%\textcolor{red}{Change following result to proposition?}

As a consequence of Theorem~\ref{psib}, we now show that the segment $s$ appearing in a label $s_\mathcal{D} \in \psi(B)$ may be expressed in a unique way as a composition of segments in $\{s_1, \ldots, s_k\}$. 

\begin{lemma}\label{Lemma_ker_coker}
Given $B \in \text{Bic}(T)$ and any $s_{\mathcal{D}} \in \psi(B)$, there is a unique way to express $s$ as a composition of some subset of the segments $s_i := \lambda(B_i, B)$ with $i = 1, \ldots, k$ where $B_1, \ldots, B_k$ are all of the biclosed sets covered by $B$. In particular, for any $s_{\mathcal{D}}, t_{\mathcal{D}^\prime} \in \psi(B)$ where $s = t\circ t^\prime$, we have that $t^\prime_{\mathcal{D}^{\prime\prime}} \in \psi(B)$ for some $\mathcal{D}^{\prime\prime} \subset \text{Seg}(T)$.
\end{lemma}
\begin{proof}
We know from Theorem~\ref{psib} that $s$ may be expressed as a composition of the elements $s_i := \lambda(B_i, B)$. Now suppose $s = s_{i_1}\circ \cdots \circ s_{i_l}$ and $s = s_{j_1} \circ \cdots \circ s_{j_m}$ for some $s_{i_1},\ldots, s_{i_1}, s_{j_1}, \ldots, s_{j_m} \in \{s_1, \ldots, s_k\}$. If $s_{i_l} \subseteq s_{j_m}$, then $s_{j_m} = s_{i_l} \circ t^\prime$ for some segment $t^\prime \in \text{Seg}(T)$. However, such an equation contradicts Lemma~\ref{nosplit}. The analogous argument shows $s_{j_m}$ is not properly contained in $s_{i_l}$. We conclude that $s_{j_m} = s_{i_l}$. By repeating this argument and removing pairs of equal segments $s_{j_n} = s_{i_r}$ with $n \le m$ and $r \le l$, we either obtain an equation $s_{j_1}\circ \cdots \circ s_{j_{n-1}} = s_{i_1}$ or $s_{j_1} = s_{i_1}\circ \cdots \circ s_{i_{m-1}}$. In either case, we reach a contradiction.
\end{proof}

\begin{remark}\label{remark_same_N}
As is mentioned in the proof of Theorem~\ref{psib}, given any two labels $s_\mathcal{D}, s_{\mathcal{D}^\prime} \in \psi(B)$, the sets $\mathcal{D}$ and $\mathcal{D}^\prime$ contain exactly the same non-faultline splits of $s$. In addition, for any biclosed set $B \in \text{Bic}(T)$, the set $\widetilde{\lambda}_{\downarrow}(B)$ consists of exactly the elements $s_\mathcal{D} \in \psi(B)$ with the property that there does not exist $s^\prime_{\mathcal{D}^\prime} \in \psi(B)$ where $s = s^\prime \circ t^\prime$ for some $t^\prime \in \text{Seg}(T)$.
\end{remark}

%??? Let $\{s_{i_1}\circ \cdots \circ s_{i_j}, s_{i_{j+1}}\circ \cdots \circ s_{i_\ell}\}$ be any faultline break of $s_{i_1}\circ \cdots \circ s_{i_\ell}$. Since $B^{(1)}, B^{(2)} \in \text{Bic}(T)$ and $B^{(1)} = B^{(2)}\backslash \{s_{i_{1}}\circ \cdots \circ s_{i_\ell}\}$, either $s_{i_1}\circ \cdots \circ s_{i_j} \in B^{(1)}$ or $s_{i_{j+1}}\circ \cdots \circ s_{i_\ell} \in B^{(1)}$.

%Let $\psi(B), \psi(B^\prime) \in \Psi(\text{Bic}(T))$ and let $B_1, \ldots, B_k \in \text{Bic}(T)$ (resp., $B^\prime_1, \ldots, B^\prime_l \in \text{Bic}(T)$) be the lower covers of $B$ (resp., $B^\prime$). Now set $s_j := \lambda(B_j,B)$ for $j = 1, \ldots, k$ and $t_j := \lambda(B^\prime_j, B^\prime)$ for $j = 1, \ldots, l$.

\subsection{Lattice structure of the shard intersection order} Let $\psi(B), \psi(B^\prime) \in \Psi(\text{Bic}(T))$ and let $B_1, \ldots, B_k \in \text{Bic}(T)$ (resp., $B^\prime_1, \ldots, B^\prime_l \in \text{Bic}(T)$) be all of the biclosed sets covered by $B$ (resp., $B^\prime$). As in the previous section, set $s_j := \lambda(B_j,B)$ for $j = 1, \ldots, k$ and $t_j := \lambda(B^\prime_j, B^\prime)$ for $j = 1, \ldots, l$. Consider $\psi(B) \cap \psi(B^\prime)$, and let $\text{Seg}(\psi(B)\cap \psi(B^\prime))$ denote the set of segments that appear in some label in the set $\psi(B)\cap \psi(B^\prime)$. 

We will prove that $\Psi(\text{Bic}(T))$ is a lattice by showing that $\psi(B) \cap \psi(B^\prime)$ is the meet of $\psi(B)$ and $\psi(B^\prime)$. This also shows that $\Psi(\text{Bic}(T))$ is a meet-subsemilattice of the Boolean lattice on the elements of $\mathcal{S}_T$. We first use Lemma~\ref{Lemma_ker_coker} and Remark~\ref{remark_same_N} to make the following important observation.

\begin{lemma}\label{intersection_ker_coker}
If $s_\mathcal{D}, s^\prime_{\mathcal{D}^\prime} \in \psi(B) \cap \psi(B^\prime)$ with $s = s^\prime \circ t^{\prime\prime}$ for some $t^{\prime\prime} \in \text{Seg}(T)$, then there exists $\mathcal{D}^{\prime\prime} \subset \text{Seg}(T)$ such that $t^{\prime\prime}_{\mathcal{D}^{\prime\prime}} \in \psi(B)\cap \psi(B^\prime)$.
\end{lemma}
\begin{proof}
By Lemma~\ref{Lemma_ker_coker}, we know that there exists $\mathcal{D}^1, \mathcal{D}^2 \subset \text{Seg}(T)$ such that $t^{\prime\prime}_{\mathcal{D}^1} \in \psi(B)$ and $t^{\prime\prime}_{\mathcal{D}^2} \in \psi(B^\prime).$ By Remark~\ref{remark_same_N}, this implies that $s^\prime$ and $t^{\prime\prime}$ are faultline splits of $s = s_{i_1}\circ \cdots \circ s_{i_n}$ and $s = t_{j_1}\circ \cdots \circ t_{j_m}$. %This implies that $t^{\prime\prime} = \lambda(B^1, B^2)$ for some $B^1, B^2 \in \left[\bigwedge_{i =1}^k B_i, B\right]$ and $t^{\prime\prime} = \lambda(B^{(1)}, B^{(2)})$ for some $B^{(1)}, B^{(2)} \in \left[\bigwedge_{j=1}^l B^\prime_j, B^\prime\right].$ 

%We can therefore write $t^{\prime\prime} = s_{i_k}\circ \cdots \circ s_{i_n}$ and $t^{\prime\prime} = t_{j_{r}}\circ \cdots \circ t_{j_m}$ for some $1 < k \le n$ and some $1 < r \le m$.

Now define $$\mathcal{D}^{\prime\prime} := \left\{\begin{array}{lcl} t^* & : & \text{if $t^* \in \mathcal{D}$ and $t^*$ is a proper subsegment of $t^{\prime\prime}$,}\\
t^* & : & \text{if $s^\prime \circ t^{*} \in \mathcal{D}$ where $t^*$ is a split of $t^{\prime\prime}$.} \end{array}\right.$$ The splits in $\mathcal{D}^{\prime\prime}$ satisfy the properties in the statement of Theorem~\ref{psib} with respect to both $\psi(B)$ and $\psi(B^\prime)$, since the splits in $\mathcal{D}$ satisfy these. Thus $t^{\prime\prime}_{\mathcal{D}^{\prime\prime}} \in \psi(B) \cap \psi(B^\prime)$.
\end{proof}

%\begin{lemma}\label{segment_factoring_lemma}
%\textcolor{purple}{Reformulate? Let \textcolor{red}{wrong subscript $t_{\mathcal{D}(t)} \in \psi(B) \cap \psi(B^\prime)$} and assume that $t = t^1\circ t^2$ for some $t^1 \in \mathcal{D}^1$ and $t^2 \in \mathcal{D}^2$ for some $s^1_{\mathcal{D}^1}, s^2_{\mathcal{D}^2} \in \psi(B)\cap \psi(B^\prime)$. Then $t^1, t^2 \in \text{Seg}(\psi(B)\cap \psi(B^\prime)).$}
%\end{lemma}
%\begin{proof}
%
%\end{proof}

%\textcolor{red}{Example of taking the intersection of two sets $\psi(B)$ and $\psi(B^\prime)$.}
%would need big example

To show that $\psi(B)\cap \psi(B^\prime) \in \Psi(\text{Bic}(T))$ we must construct a biclosed set $B^{\prime\prime}$ that satisfies $\psi(B^{\prime\prime}) = \psi(B) \cap \psi(B^\prime)$. With this goal in mind, we let $\{s^{(i)}_{\mathcal{D}^{(i)}}\}_{i = 1}^\ell$ denote the elements of $\psi(B)\cap \psi(B^\prime)$ where $s^{(i)}$ appears in exactly one label in $\psi(B)\cap \psi(B^\prime)$. If such a collection of labels exists, it is unique. We prove that such a collection exists in the following lemma. After that, we work toward showing that $\psi\left(\bigvee_{i=1}^\ell J(s^{(i)}_{\mathcal{D}^{(i)}})\right) = \psi(B)\cap \psi(B^\prime)$, which is the content of Lemma~\ref{meet_in_Psi}.

\begin{lemma}\label{nonempty_s^i}
The collection of elements $\{s^{(i)}_{\mathcal{D}^{(i)}}\}_{i = 1}^\ell \subset \psi(B) \cap \psi(B^\prime)$ is nonempty. In particular, the segments $s^{(1)}, \ldots, s^{(\ell)}$ are pairwise distinct.
\end{lemma}
\begin{proof}
Suppose that any segment in $\text{Seg}(\psi(B)\cap \psi(B^\prime))$ belongs to at least two labels of $\psi(B) \cap \psi(B^\prime)$. Let $s_\mathcal{D}$ and $s_{\mathcal{D}^\prime}$ be minimal elements of the subposet $\psi(B)\cap \psi(B^\prime)$ of $\mathcal{S}_T$. Since $\mathcal{D} \neq \mathcal{D}^\prime$, there exists $t \in \mathcal{D}$ and $t^\prime \in \mathcal{D}^\prime$ such that $s = t \circ t^\prime$. From Remark~\ref{remark_same_N}, segments $t$ and $t^\prime$ are faultline splits of $s$ in terms of its expression as a composition of segments in $\{s_1,\ldots, s_k\}$. Segments $t$ and $t^\prime$ are also faultline splits of $s$ in terms of its expression as a composition of segments in $\{t_1, \ldots, t_l\}$. In particular, $t, t^\prime \in \overline{\{s_1, \ldots, s_k\}}\cap \overline{\{t_1, \ldots, t_l\}}.$

We now show that $t_{\mathcal{D}(t)} \in \psi(B)\cap \psi(B^\prime)$ where $\mathcal{D}(t) := \{s^\prime \in J(s_\mathcal{D}): \text{$s^\prime$ is a split of $t$}\}.$ This contradicts that $s_\mathcal{D}$ is a minimal element of the subposet $\psi(B)\cap \psi(B^\prime)$ and therefore completes the proof. We show that $t_{\mathcal{D}(t)} \in \psi(B)$ and an analogous argument shows that $t_{\mathcal{D}(t)} \in \psi(B^\prime)$ so we omit the latter.

First, consider the biclosed sets $$B^1 := \left(\bigwedge_{i=1}^k B_i\right) \vee (J(s_\mathcal{D})\backslash \{s\}), \ B^2 := \left(\bigwedge_{i=1}^k B_i\right) \vee (J(t_{\mathcal{D}(t)})\backslash \{t\}) \in \left[\bigwedge_{i=1}^k B_i, B\right].$$ We know that $B^2 \in \left[\bigwedge_{i=1}^k B_i, B\right]$ since $J(t_{\mathcal{D}(t)})\backslash \{t\} < J(s_\mathcal{D})\backslash \{s\}$ by Lemma~\ref{Lemma_join_irr_containment}. Note that $t \not \in B^2$ since $t \in \overline{\{s_1, \ldots, s_k\}}\cap \overline{\{t_1, \ldots, t_l\}}.$  We show that $\widetilde{\lambda}(B^2, B^2 \sqcup \{t\}) = t_{\mathcal{D}(t)}$ by showing that $B^2 \sqcup \{t\}$ is biclosed. From this it follows that $t_{\mathcal{D}(t)} \in \psi(B)$.

By construction, $B^2\sqcup \{t\}$ is coclosed. Since $B^2$ is closed, we show that $B^2\sqcup \{t\}$ is closed by showing that, without loss of generality, for any $t^1 \in \bigwedge_{i = 1}^k B_i$ that is composable with $t$, we have $t\circ t^1 \in \bigwedge_{i = 1}^k B_i$. We know that $t, t^1 \in B$ so $t \circ t^1 \in B.$ 

Now, suppose that $t\circ t^1 \in \overline{\{s_1,\ldots, s_k\}}$. This means there exists $\mathcal{D}^{*} \subset \text{Seg}(T)$ such that $(t \circ t^1)_{\mathcal{D}^*} \in \psi(B)$. By Lemma~\ref{Lemma_ker_coker}, we know that there is a unique way to express $t,$ and $t\circ t^1$ as a composition of elements of $\overline{\{s_1,\ldots, s_k\}}$. It follows that $t^1 \in \overline{\{s_1, \ldots, s_k\}}$, but this contradicts that $t^1 \in \bigwedge_{i = 1}^k B_i$. \end{proof}
%Thus there exists sets $\mathcal{D}^1, \mathcal{D}^2 \subset \text{Seg}(T)$ such that $t_{\mathcal{D}^1}, t^\prime_{\mathcal{D}^2} \in \psi(B) \cap \psi(B^\prime)$.

The following proposition is a crucial step in proving the lattice property of $\Psi(\text{Bic}(T))$. The reader should compare the statement with the characterization of elements of $\widetilde{\lambda}_{\downarrow}(B)$ in Remark~\ref{remark_same_N}.

%With the knowledge of Remark~\ref{remark_same_N}, the following lemma is crucial to proving Lemma~\ref{meet_in_Psi}.

\begin{proposition}\label{Lemma_property_of_s_i}
The labels $\{s^{(i)}_{\mathcal{D}^{(i)}}\}_{i = 1}^\ell$ are precisely the elements $s_\mathcal{D} \in \psi(B)\cap \psi(B^\prime)$ with the property that there does not exist $s^\prime_{\mathcal{D}^\prime} \in \psi(B) \cap \psi(B^\prime)$ where $s = s^\prime \circ t^\prime$ for some $t^\prime \in \text{Seg}(T)$.
\end{proposition}
\begin{proof}%[Proof of Lemma~\ref{Lemma_property_of_s_i}]
Suppose there exists a label $s^\prime_{\mathcal{D}^\prime} \in \psi(B) \cap \psi(B^\prime)$ such that $s^{(i)} = s^\prime \circ t^\prime$ for some $t^\prime \in \text{Seg}(T)$. By Lemma~\ref{intersection_ker_coker}, there exists $\mathcal{D}^{\prime\prime} \subset \text{Seg}(T)$ such that $t^{\prime}_{\mathcal{D}^{\prime\prime}} \in \psi(B)\cap \psi(B^\prime)$. By Remark~\ref{remark_same_N}, $s^\prime$ and $t^\prime$ are faultline splits of $s^{(i)}$ regardless of whether this segment is expressed in terms of $s_1, \ldots, s_k$ or $t_1, \ldots, t_l$. Without loss of generality, we assume that $s^\prime \in \mathcal{D}^{(i)}$. Thus $s^{(i)}_{\mathcal{D}^{(i)}}, s^{(i)}_{(\mathcal{D}^{(i)}\backslash\{s^\prime\}) \sqcup \{t^\prime\}} \in \psi(B)\cap \psi(B^\prime)$. However, this contradicts the definition of $s^{(i)}_{\mathcal{D}^{(i)}}$.

Let $s_{\mathcal{D}} \in \psi(B)\cap \psi(B^\prime)$ be a label with the stated property. Suppose there are two distinct labels $s_{\mathcal{D}}, s_{\mathcal{D}^\prime} \in \psi(B)\cap \psi(B^\prime)$. Then there exist $t \in \mathcal{D}$ and $t^\prime \in \mathcal{D}^\prime$ such that $s = t \circ t^\prime$. We claim that $t_{\mathcal{D}(t)}, t^\prime_{\mathcal{D}^\prime(t^\prime)} \in  \psi(B)\cap \psi(B^\prime)$ where $\mathcal{D}(t) := \{s^\prime \in J(s_\mathcal{D}): \text{$s^\prime$ is a split of $t$}\}$ and $\mathcal{D}^\prime(t^\prime) := \{s^\prime \in J(s_{\mathcal{D}^\prime}): \text{$s^\prime$ is a split of $t^\prime$}\}.$ This claim contradicts that $s_{\mathcal{D}}$ has the property in the statement of the lemma. Thus $s$ appears in exactly one label in $\psi(B)\cap \psi(B^\prime)$. Therefore, $s = s^{(i)}$ and $\mathcal{D}= \mathcal{D}^{(i)}$ for some $i = 1, \ldots, \ell$.

One proves our claim by adapting the proof of Lemma~\ref{nonempty_s^i}.
\end{proof}

%\#\#\#
%
%Now let $(s^{(i)})_{\mathcal{D}^{(i)}}$ for $i = 1, \ldots, \ell$ denote the elements of $\psi(B)\cap \psi(B^\prime)$ with the following properties:
%\begin{itemize}
%\item label $(s^{(i)})_{\mathcal{D}^{(i)}}$ is \textbf{locally minimal} in the sense that it does not cover any other element of $\psi(B)\cap\psi(B^\prime)$ in the poset $\mathcal{S}_T$, and 
%\item for any label $(s^{(i)})_{\mathcal{D}^{(i)}}$, the segment $s^{(i)}$ may not be expressed as the composition of at least two segments in other locally minimal elements of $\psi(B)\cap\psi(B^\prime).$
%\end{itemize}
%If we assume that $\psi(B)\cap\psi(B^\prime) \neq \emptyset$, the collection of labels $(s^{(i)})_{\mathcal{D}^{(i)}}$ for $i = 1, \ldots, \ell$ is non-empty as it contains the minimal elements of $\psi(B)\cap\psi(B^\prime)$.
%Let $\text{Seg}(\psi(B)\cap \psi(B^\prime))$ denote the set of segments that appear in some label in the set $\psi(B)\cap \psi(B^\prime)$.
%
%\#\#\#

\begin{lemma}\label{Lemma_seg_of_int}
$\text{Seg}(\psi(B)\cap \psi(B^\prime)) = \overline{\{s^{(1)}, \ldots, s^{(\ell)}\}}$. %Moreover, any element of $\text{Seg}(\psi(B)\cap \psi(B^\prime))$ has a unique expression as a composition of the segments $s^{(1)}, \ldots, s^{(\ell)}$.
\end{lemma}
\begin{proof}
To show that $\overline{\{s^{(1)}, \ldots, s^{(\ell)}\}} \subset \text{Seg}(\psi(B)\cap \psi(B^\prime))$, it is enough to show that the composition of any two composable segments $s^{(i)}$ and $s^{(j)}$ belongs to $\text{Seg}(\psi(B)\cap \psi(B^\prime))$. We know that $s^{(i)}\circ s^{(j)} \in \overline{\{s_1, \ldots, s_k\}} \cap \overline{\{t_1, \ldots, t_l\}}$ since this is true for $s^{(i)}$ and $s^{(j)}$. Now from the description of $\psi(B)$ and $\psi(B^\prime)$ obtained in Theorem~\ref{psib}, we see that $s^{(i)}\circ s^{(j)} \in \text{Seg}(\psi(B)\cap \psi(B^\prime))$.

To prove the opposite containment, let $s \in \text{Seg}(\psi(B) \cap \psi(B^\prime))$, and assume that any segment in $\text{Seg}(\psi(B)\cap \psi(B^\prime))$ that is shorter than $s$ belongs to $\overline{\{s^{(1)}, \ldots, s^{(\ell)}\}}$. Since $s \in \text{Seg}(\psi(B) \cap \psi(B^\prime))$, there exists $\mathcal{D} \subset \text{Seg}(T)$ such that $s_{\mathcal{D}} \in \psi(B) \cap \psi(B^\prime)$.

We can further assume that $s \neq s^{(i)}$ for any $i = 1, \ldots, \ell$. By Proposition~\ref{Lemma_property_of_s_i}, this implies that there exists $s^{\prime}_{\mathcal{D}^{\prime}} \in \psi(B)\cap \psi(B^\prime)$ where $s = s^{\prime}\circ t^{\prime\prime}$ for some $t^{\prime\prime} \in \text{Seg}(T)$. By Lemma~\ref{intersection_ker_coker}, there exists $\mathcal{D}^{\prime\prime} \subset \text{Seg}(T)$ such that $t^{\prime\prime}_{\mathcal{D}^{\prime\prime}} \in \psi(B)\cap \psi(B^\prime)$. Since $s^\prime$ and $t^{\prime\prime}$ are both shorter than $s$, we have that $s^\prime, t^{\prime\prime} \in \overline{\{s^{(1)}, \ldots, s^{(\ell)}\}}$ so $s = s^\prime\circ t^{\prime\prime} \in \overline{\{s^{(1)}, \ldots, s^{(\ell)}\}}$. \end{proof}

\begin{lemma}\label{Lemma_cjr_for_meet}
{The expression $\bigvee_{i=1}^\ell J(s^{(i)}_{\mathcal{D}^{(i)}})$ is a canonical join representation.}
\end{lemma}
\begin{proof}
By Theorem~\ref{Thm_CJC}, it is enough to show that for any $s^{(i)}_{\mathcal{D}^{(i)}}, s^{(j)}_{\mathcal{D}^{(j)}} \in \psi(B) \cap \psi(B^\prime)$ with $i \neq j$ the expression $J(s^{(i)}_{\mathcal{D}^{(i)}})\vee J(s^{(j)}_{\mathcal{D}^{(j)}})$ is a canonical join representation. It follows from Lemma~\ref{nonempty_s^i} that $s^{(i)} \neq s^{(j)}$ so we need to show that $s^{(i)}_{\mathcal{D}^{(i)}}$ and $s^{(j)}_{\mathcal{D}^{(j)}}$ satisfy properties 2) and 3) from Theorem~\ref{Thm_CJC}.

First, we show $s^{(i)}_{\mathcal{D}^{(i)}}$ and $s^{(j)}_{\mathcal{D}^{(j)}}$ satisfy property 2) by showing that, without loss of generality, $s^{(i)}$ is not expressible as a composition of at least two segments from $J(s^{(i)}_{\mathcal{D}^{(i)}})\cup J(s^{(j)}_{\mathcal{D}^{(j)}})$. Suppose we can write $s^{(i)} = t^1\circ \cdots \circ  t^r$ where $t^1, \ldots, t^r \in J(s^{(i)}_{\mathcal{D}^{(i)}})\cup J(s^{(j)}_{\mathcal{D}^{(j)}})$ with $r \ge 2$. Clearly, $t^1, t^2\circ \cdots \circ t^r \in J(s^{(i)}_{\mathcal{D}^{(i)}}) \vee J(s^{(j)}_{\mathcal{D}^{(j)}})$. 

We also know that $J(s^{(i)}_{\mathcal{D}^{(i)}}) \vee J(s^{(j)}_{\mathcal{D}^{(j)}}) \le B$ and $J(s^{(i)}_{\mathcal{D}^{(i)}}) \vee J(s^{(j)}_{\mathcal{D}^{(j)}}) \le B^\prime.$ To see this, note that $s^{(i)}_{\mathcal{D}^{(i)}}, s^{(j)}_{\mathcal{D}^{(j)}} \in \psi(B)$ implies that there exists $(B^1, B^2), (B^3, B^4) \in \text{Cov}([\bigwedge_{i =1}^k, B_i, B])$ such that $\widetilde{\lambda}(B^1, B^2) = s^{(i)}_{\mathcal{D}^{(i)}}$ and $\widetilde{\lambda}(B^3, B^4) = s^{(j)}_{\mathcal{D}^{(j)}}$. By Proposition~\ref{Prop_join_irr_des}, we have that $J(s^{(i)}_{\mathcal{D}^{(i)}}) \le B^2$ and $J(s^{(j)}_{\mathcal{D}^{(j)}}) \le B^4.$ Thus $J(s^{(i)}_{\mathcal{D}^{(i)}}) \vee J(s^{(j)}_{\mathcal{D}^{(j)}}) \le B$. The proof that $J(s^{(i)}_{\mathcal{D}^{(i)}}) \vee J(s^{(j)}_{\mathcal{D}^{(j)}}) \le B^\prime$ is similar.

%Since $s^{(i)}_{\mathcal{D}^{(i)}}, s^{(j)}_{\mathcal{D}^{(j)}} \in \psi(B)\cap \psi(B^\prime)$. 

By Remark~\ref{remark_same_N}, we obtain that $t^1$ and $t^2\circ \cdots \circ t^r$ are faultline splits of $s^{(i)} = s_{i_1}\circ \cdots \circ s_{i_n}$ and $s^{(i)} = t_{j_1} \circ \cdots \circ t_{j_m}$. Therefore, $t^1, t^2\circ \cdots \circ t^r \in \text{Seg}(\psi(B)\cap \psi(B^\prime)).$ We can assume, without loss of generality, that $t^1 \in \mathcal{D}^{(i)}$. We obtain that $s^{(i)}_{\mathcal{D}^{(i)}}, s^{(i)}_{(\mathcal{D}^{(i)}\backslash\{t^1\}) \sqcup \{t^2\circ \cdots \circ t^r\}} \in \psi(B)\cap \psi(B^\prime)$. However, this contradicts that $s^{(i)}$ is a segment in exactly one label in $\psi(B) \cap \psi(B^\prime)$. 

Lastly, we show $s^{(i)}_{\mathcal{D}^{(i)}}$ and $s^{(j)}_{\mathcal{D}^{(j)}}$ satisfy property 3). Assume that we have $s^{(i)} \subseteq s^{(j)}$. Write $s^{(j)} = s_{i_1}\circ \cdots \circ s_{i_n}$ for some $s_{i_1}, \ldots, s_{i_n} \in \{s_1, \ldots, s_k\}$ and $s^{(j)} = t_{j_1} \circ \cdots \circ t_{j_m}$ for some $t_{j_1}, \ldots, t_{j_m} \in \{t_1, \ldots, t_l\}$. One checks that there are the following three cases describing how $s^{(i)}$ may be expressed in terms of $s_{i_1}, \ldots, s_{i_n}$ and $t_{j_1}, \ldots, t_{j_m}$:
\begin{itemize}
\item[i)] $s^{(i)} = s_{i_\ell}\circ \cdots \circ s_{i_{\ell^\prime}}$ and $s^{(i)} = t^r\circ t_{j_{r+1}}\circ \cdots \circ t_{j_{r^\prime-1}}\circ t^{r^\prime},$
\item[ii)] $s^{(i)} = t_{j_r}\circ \cdots \circ t_{j_{r^\prime}}$ and $s^{(i)} = t^\ell \circ s_{i_{\ell+1}}\circ \cdots \circ s_{i_{\ell^\prime-1}}\circ t^{\ell^\prime},$
\item[iii)] $s^{(i)} = t^\ell \circ s_{i_{\ell+1}}\circ \cdots \circ s_{i_{\ell^\prime-1}}\circ t^{\ell^\prime}$ and $s^{(i)} = t^r\circ t_{j_{r+1}}\circ \cdots \circ t_{j_{r^\prime-1}}\circ t^{r^\prime}.$
\end{itemize}
In these cases, $t^r$ is a split of $t_{j_r}$, $t^{r^\prime}$ is a split of $t_{j_{r^\prime}}$, $t^\ell$ is a split of $s_{i_\ell}$, and $t^{\ell^\prime}$ is a split of $s_{i_{\ell^\prime}}$. In the degenerate case when $r = r^\prime$ (resp., $\ell = \ell^\prime$), we mean that $s^{(i)}$ is proper subsegment of $t_{j_r}$ (resp., $s_{i_\ell}$) that is not a split of $t_{j_r}$ (resp., $s_{i_\ell}$).

Now write $s^{(i)} \subseteq s_{i_\ell}\circ \cdots \circ s_{i_{\ell^\prime}}$ with $\ell^\prime - \ell$ as small as possible. We can assume without loss of generality that we are in case ii) or iii). This means we can express $s^{(i)}$ in terms of $\{s_1, \ldots, s_k\}$ as $s^{(i)} = s_{k_1}\circ \cdots \circ s_{k_r}$ where $s_{k_1}$ is not a split of $s_{i_\ell}$ or $s_{i_{\ell^\prime}}$ and no segment in $\{s_{i_\ell}, \ldots, s_{i_{\ell^\prime}}\}$ is a split of $s_{k_1}$. Similarly, segment $s_{k_r}$ will have these same properties. It follows that $s_{k_1}, s_{k_r} \not \in \{s_{i_1}, \ldots, s_{i_n}\}$. 

%Now write $s^{(i)} \subseteq s_{i_\ell}\circ \cdots \circ s_{i_{\ell^\prime}}$ with $\ell^\prime - \ell$ as small as possible. We can assume without loss of generality that we are in case ii) or iii). This means, we can express $s^{(i)}$ in terms of $\{s_1, \ldots, s_k\}$ as $s^{(i)} = s_{k_1}\circ \cdots \circ s_{k_r}$ where $s_{k_1}$ is not a split of $s_{i_\ell}$ or $s_{i_{\ell^\prime}}$ and no segment in $\{s_{i_\ell}, \ldots, s_{i_{\ell^\prime}}\}$ is a split of $s_{k_1}$. Similarly, segment $s_{k_r}$ will have these same properties. We see that $s_{k_1}, s_{k_r} \not \in \{s_{i_1}, \ldots, s_{i_n}\}$. 

We claim that $\{s_{i_1}, \ldots, s_{i_n}\}\cap \{s_{k_1},\ldots, s_{k_r}\} = \emptyset$. If this were not true, then from the properties of $s_{k_1}$ and $s_{k_r}$, there must exist $s_{k_j} \in \{s_{k_1},\ldots, s_{k_r}\}$ that is either a split of an element in $\{s_{i_\ell},\ldots, s_{i_{\ell^\prime}}\}$ or an element of this set is a split of $s_{k_j}$. The existence of such an element contradicts Lemma~\ref{nosplit}. 

We obtain that $J({s_{k_j}}_{\mathcal{D}_{k_j}}) \not \le \bigvee_{r = 1}^n J({s_{i_r}}_{\mathcal{D}_{i_r}})$ for any $s_{k_j} \in \{s_{k_1}, \ldots, s_{k_r}\}$. If this containment did hold, then $\bigvee_{i = 1}^k J({s_i}_{\mathcal{D}_i})$ would not be an irredundant join representation. This would contradict that $\bigvee_{i = 1}^k J({s_i}_{\mathcal{D}_i})$ is a canonical join representation.

Next, by Lemma~\ref{Lemma_special_subseg}, there exists $s_{k^*} \in \{s_{k_1}, \ldots, s_{k_r}\}$ such that $s_{k^*} \in J(s^{(i)}_{\mathcal{D}^{(i)}}).$ Observe that $J({s_{k^*}}_{\mathcal{D}_{k^*}}) \le J(s^{(i)}_{\mathcal{D}^{(i)}})$ by Lemma~\ref{Lemma_join_irr_containment}. This implies that $J(s^{(i)}_{\mathcal{D}^{(i)}}) \not \le \bigvee_{r = 1}^n J({s_{i_r}}_{\mathcal{D}_{i_r}})$. {Since $J(s^{(j)}_{\mathcal{D}^{(j)}}) \le \bigvee_{r = 1}^n J({s_{i_r}}_{\mathcal{D}_{i_r}})$,} we see that $J(s^{(i)}_{\mathcal{D}^{(i)}}) \not \le J(s^{(j)}_{\mathcal{D}^{(j)}})$. 

Since $s^{(i)}\subseteq s^{(j)}$, it must be a proper subsegment of $s^{(j)}$, and it is not a split of $s^{(j)}$. Thus $J(s^{(j)}_{\mathcal{D}^{(j)}}) \not \le J(s^{(i)}_{\mathcal{D}^{(i)}}),$ since $s^{(j)} \not \in J(s^{(i)}_{\mathcal{D}^{(i)}})$. We obtain that $s^{(i)}_{\mathcal{D}^{(i)}}$ and $s^{(j)}_{\mathcal{D}^{(j)}}$ satisfy property 3). We conclude that the expression $J(s^{(i)}_{\mathcal{D}^{(i)}}) \vee J(s^{(j)}_{\mathcal{D}^{(j)}})$ is a canonical join representation.\end{proof}

%%We obtain that $J({s_{k_1}}_{\mathcal{D}_{k_1}}) \not \le \bigvee_{r = 1}^n J({s_{i_r}}_{D_{i_r}})$. If this containment did hold, it would contradict that $\bigvee_{i = 1}^k J({s_i}_{D_i})$ is a canonical join representation. 

\begin{lemma}\label{Lemma_special_subseg}
If $J(s_\mathcal{D}) \in \text{JI}(\text{Bic}(T))$ and $s = s_{k_1}\circ \cdots \circ s_{k_r}$, then there exists $s_{k^*} \in \{s_{k_1}, \ldots, s_{k_r}\}$ such that $s_{k^*} \in J(s_{\mathcal{D}})$.
\end{lemma}
\begin{proof}
If $s_{k_r} \not \in J(s_\mathcal{D})$, then $s_{k_1}\circ \cdots \circ s_{k_{r-1}} \in J(s_{\mathcal{D}})$. By Lemma~\ref{Lemma_join_irr_containment}, we have that $$J((s_{k_1}\circ \cdots \circ s_{k_{r-1}})_{\mathcal{D}(s_{k_1}\circ \cdots \circ s_{k_{r-1}})}) \le J(s_\mathcal{D})$$ where ${\mathcal{D}(s_{k_1}\circ \cdots \circ s_{k_{r-1}})} = \{s^\prime \in J(s_\mathcal{D}): \ \text{$s^\prime$ is a split of $s_{k_1}\circ \cdots \circ s_{k_{r-1}}$}\}.$ If $s_{k_{r-1}} \not \in J(s_{\mathcal{D}})$, then $s_{k_{r-1}} \not \in J((s_{k_1}\circ \cdots \circ s_{k_{r-1}})_{\mathcal{D}(s_{k_1}\circ \cdots \circ s_{k_{r-1}})})$ so $s_{k_1}\circ \cdots \circ s_{k_{r-2}} \in J((s_{k_1}\circ \cdots \circ s_{k_{r-1}})_{\mathcal{D}(s_{k_1}\circ \cdots \circ s_{k_{r-1}})})$. Now by Lemma~\ref{Lemma_join_irr_containment}, we have that $$J((s_{k_1}\circ \cdots \circ s_{k_{r-2}})_{\mathcal{D}(s_{k_1}\circ \cdots \circ s_{k_{r-2}})}) \le J(s_\mathcal{D})$$ where ${\mathcal{D}(s_{k_1}\circ \cdots \circ s_{k_{r-2}})} = \{s^\prime \in J(s_\mathcal{D}): \ \text{$s^\prime$ is a split of $s_{k_1}\circ \cdots \circ s_{k_{r-2}}$}\}.$ Continuing this process, we obtain that there exists $s_{k^*} \in  J(s_{\mathcal{D}})$.
\end{proof}

\begin{lemma}\label{meet_in_Psi}
{We have that $\psi\left(\bigvee_{i=1}^\ell J(s^{(i)}_{\mathcal{D}^{(i)}})\right) = \psi(B)\cap \psi(B^\prime).$}
\end{lemma}
\begin{proof}
By Lemma~\ref{Lemma_cjr_for_meet}, we know that $\widetilde{\lambda}_{\downarrow}\left(\bigvee_{i=1}^\ell J(s^{(i)}_{\mathcal{D}^{(i)}})\right) = \{s^{(i)}_{\mathcal{D}^{(i)}}\}_{i = 1}^\ell.$ Thus, by Theorem~\ref{psib}, any element of $\psi\left(\bigvee_{i=1}^\ell J(s^{(i)}_{\mathcal{D}^{(i)}})\right)$ is of the form $(s^{(i_1)}\circ\cdots \circ s^{(i_r)})_\mathcal{D}$ for some $i_1, \ldots, i_r \in \{1,\ldots, \ell\}$. From Lemma~\ref{Lemma_seg_of_int}, we know that $s^{(i_1)}\circ\cdots \circ s^{(i_r)} \in \text{Seg}(\psi(B)\cap \psi(B^\prime))$. 

We now show that $(s^{(i_1)}\circ\cdots \circ s^{(i_r)})_\mathcal{D} \in \psi(B)\cap \psi(B^\prime)$. We show $(s^{(i_1)}\circ\cdots \circ s^{(i_r)})_\mathcal{D} \in \psi(B)$ and the analogous argument shows that $(s^{(i_1)}\circ\cdots \circ s^{(i_r)})_\mathcal{D} \in \psi(B^\prime)$. Consider the biclosed set $$B^* := \left(\bigwedge_{i=1}^kB_i\right) \vee \left(J((s^{(i_1)}\circ \cdots \circ s^{(i_r)})_\mathcal{D})\backslash \{s^{(i_1)}\circ \cdots \circ s^{(i_r)}\}\right).$$ Clearly, $\bigwedge_{i=1}^kB_i \le B^*$ and $B^* < B$ as $s^{(i_1)}\circ \cdots \circ s^{(i_r)} \in B\backslash B^*$. 

If we assume that $B^*\sqcup \{s^{(i_1)}\circ \cdots \circ s^{(i_r)}\}$ is biclosed, we obtain that $\widetilde{\lambda}(B^*, B^*\sqcup \{s^{(i_1)}\circ \cdots \circ s^{(i_r)}\}) = (s^{(i_1)}\circ\cdots \circ s^{(i_r)})_\mathcal{D}$ and $B^*\sqcup \{s^{(i_1)}\circ \cdots \circ s^{(i_r)}\} \le B$. From this it follows that $(s^{(i_1)}\circ\cdots \circ s^{(i_r)})_\mathcal{D} \in \psi(B)$. One shows that $B^*\sqcup \{s^{(i_1)}\circ \cdots \circ s^{(i_r)}\}$ is biclosed by adapting the argument that we used in Lemma~\ref{nonempty_s^i}.% Therefore, to complete the proof we show that $B^*\sqcup \{s^{(i_1)}\circ \cdots \circ s^{(i_r)}\}$ is biclosed.

%By construction, $B^*\sqcup \{s^{(i_1)}\circ \cdots \circ s^{(i_r)}\}$ is coclosed. Since $B^*$ is closed, we show that $B^*\sqcup \{s^{(i_1)}\circ \cdots \circ s^{(i_r)}\}$ is closed by showing that, without loss of generality, for any $t \in \bigwedge_{i = 1}^k B_i$ that is  composable with $s^{(i_1)}\circ \cdots \circ s^{(i_r)}$, we have $s^{(i_1)}\circ \cdots \circ s^{(i_r)}\circ t \in \bigwedge_{i = 1}^k B_i$. We know that $s^{(i_1)}\circ \cdots \circ s^{(i_r)}, t \in B$ so $s^{(i_1)}\circ \cdots \circ s^{(i_r)}\circ t \in B.$ 

%Now, suppose that $s^{(i_1)}\circ \cdots \circ s^{(i_r)}\circ t \in \overline{\{s_1,\ldots, s_k\}}$. This means there exists $\mathcal{D}^{*} \subset \text{Seg}(T)$ such that $(s^{(i_1)}\circ \cdots \circ s^{(i_r)}\circ t)_{\mathcal{D}^*} \in \psi(B)$. By Lemma~\ref{Lemma_ker_coker}, we know that there is a unique way to express $s^{(i_1)}, \ldots, s^{(i_r)},$ and $s^{(i_1)}\circ \cdots \circ s^{(i_r)}\circ t$ as a composition of elements of $\overline{\{s_1,\ldots, s_k\}}$. It follows that $t \in \overline{\{s_1, \ldots, s_k\}}$, but this contradicts that $t \in \bigwedge_{i = 1}^k B_i$.

%\textcolor{red}{STUFF}

%\textcolor{red}{comments about how faultline and non-faultline splits may be chosen independently.}

%It remains to show that any non-faultline split $s^{(i_1)}\circ \cdots \circ s^{(i_{j-1})}\circ t^{(i_j)} \in \mathcal{D}$ of $s^{(i_1)}\circ \cdots \circ s^{(i_r)}$ 

Conversely, suppose $(s^{(i_1)}\circ \cdots \circ s^{(i_r)})_{\mathcal{D}} \in \psi(B)\cap \psi(B^\prime)$ where this label can be expressed this way by Lemma~\ref{Lemma_seg_of_int}. By Theorem~\ref{psib}, we know that there exist labels $\widetilde{\lambda}(B^1, B^2) \in \psi\left(\bigvee_{i=1}^\ell J(s^{(i)}_{\mathcal{D}^{(i)}})\right)$ such that $\lambda(B^1, B^2) = s^{(i_1)}\circ \cdots \circ s^{(i_r)}.$ Thus it is enough to show that $\mathcal{D}$ satisfies the properties appearing in the statement of Theorem~\ref{psib} with respect to the labels $\{s^{(i)}_{\mathcal{D}^{(i)}}\}_{i = 1}^\ell$. In fact, it is enough to check that the non-faultline splits of $\mathcal{D}$ satisfy (iv) in the statement of Theorem~\ref{psib}. If $r = 1$, the label under consideration is $s^{(i_1)}_{\mathcal{D}^{(i_1)}}$ and the result follows.

%\textcolor{red}{STUFF}

Now suppose that $s^{(i_1)}\circ \cdots \circ s^{(i_{j-1})}\circ t^{(i_j)} \in \mathcal{D}$ is a non-faultline split of $s^{(i_1)}\circ \cdots \circ s^{(i_{r})}$ where $t^{(i_j)}$ is a split of $s^{(i_j)}$. We assume that the result holds for labels $(s^{(j_1)}\circ \cdots \circ s^{(j_{r^\prime})})_{\mathcal{D}} \in \psi(B)\cap \psi(B^\prime)$ with $r^\prime < r$. Using the approach in the proof of Lemma~\ref{intersection_ker_coker}, we have that the label $(s^{(i_j)}\circ \cdots \circ s^{(i_r)})_{\mathcal{D}^{\prime\prime}} \in \psi(B)\cap \psi(B^\prime)$ where $$\mathcal{D}^{\prime\prime} := \left\{\begin{array}{lcl} t^* & : & \text{if $t^* \in \mathcal{D}$ and $t^*$ is a proper subsegment of $s^{(i_j)}\circ \cdots \circ s^{(i_r)}$,}\\
t^* & : & \text{if  $s^{(i_1)}\circ \cdots \circ s^{(i_{j-1})} \circ t^{*} \in \mathcal{D}$ where $t^*$ is a split of $s^{(i_j)}\circ \cdots \circ s^{(i_r)}$.} \end{array}\right.$$

We observe that $t^{(i_j)} \in \mathcal{D}^{\prime\prime}$. By induction, $t^{(i_j)} \in \mathcal{D}^{(i_j)}$. We conclude that $(s^{(i_1)}\circ \cdots \circ s^{(i_r)})_{\mathcal{D}} \in \psi\left(\bigvee_{i=1}^\ell J(s^{(i)}_{\mathcal{D}^{(i)}})\right)$. \end{proof}

%We reduce to the case where where $j = r$. 

%If $t^{(i_j)} \not \in \mathcal{D}^{(i_j)}$, then, by writing $s^{(i_j)} = t^{(i_j)} \circ t^{j}$ for some $t^j \in \text{Seg}(T)$, \textcolor{red}{we see that $t^j \in \mathcal{D}^{(i_j)}$.} Now by Lemma~\ref{segment_factoring_lemma}, we see that $t^{(i_j)}, t^{j} \in \text{Seg}(\psi(B)\cap \psi(B^\prime))$. Since $\text{Seg}(\psi(B)\cap \psi(B^\prime)) = \overline{\{s^{(1)}, \ldots, s^{(\ell)}\}}$ by Lemma~\ref{Lemma_seg_of_int}, we have contradicted that $s^{(i_j)}$ may not be expressed as a composition of at least two elements of $\{s^{(1)}, \ldots, s^{(\ell)}\}$. We conclude that $(s^{(i_1)}\circ \cdots \circ s^{(i_r)})_{\mathcal{D}} \in \psi\left(\bigvee_{i=1}^\ell J(s^{(i)}_{\mathcal{D}^{(i)}})\right)$.
%\end{proof}

\begin{theorem}\label{Thm_lattice_prop}
The shard intersection order $\Psi(\text{Bic}(T))$ is a lattice.
\end{theorem}
\begin{proof}
The shard intersection order $\Psi(\text{Bic}(T))$ is finite. So, to prove that it is a lattice, we show that it has a unique maximal element and that for any $\psi(B), \psi(B^\prime) \in \Psi(\text{Bic}(T))$ one has $\psi(B)\cap \psi(B^\prime) \in \Psi(\text{Bic}(T))$. The latter is a consequence of Lemma~\ref{meet_in_Psi}.

Next, the shard intersection order of $\text{Bic}(T)$ has a unique maximal element. The set $\text{Seg}(T)$ is the top element of $\text{Bic}(T)$, and the coatoms of $\text{Bic}(T)$ are the sets of the form $\text{Seg}(T)\backslash\{[a,b]\}$ where $a$ and $b$ are interior vertices connected by an edge of T. Thus the meet of all coatoms of $\text{Bic}(T)$ is the empty set. Now, using the definition of $\Psi(\text{Bic}(T))$ and Proposition~\ref{Prop_join_irr_des}, we see that $\psi(\text{Seg}(T)) = \widetilde{\lambda}(\text{Cov}(\text{Bic}(T))) = \mathcal{S}_T.$ This implies that $\psi(\text{Seg}(T))$ is the unique maximal element of $\Psi(\text{Bic}(T))$.
\end{proof}

\begin{remark}
Theorem~\ref{Thm_lattice_prop} was originally conjectured by the first and second author in \cite[Conjecture 6.3]{cliftondillery}.
\end{remark}

\begin{remark}
It is an open problem in \cite[Problem 9.5]{ReadingPAB} to determine which congruence-uniform lattices $L$ have the property that $\Psi(L)$ is a lattice. Theorem~\ref{Thm_lattice_prop} provides many new examples of congruence-uniform lattices with this property.
\end{remark}

\begin{remark}
In the very recent preprint \cite[Theorem 1.1]{muhle}, M\"uhle showed that a congruence-uniform lattice $L$ must be \textbf{spherical} if its shard intersection order $\Psi(L)$ is a lattice. That is, the M\"obius function on $L$, denoted $\mu_L(-,-)$, must satisfy $\mu(\hat{0}, \hat{1}) \neq 0 $ if $\Psi(L)$ is a lattice. 

One shows that $\text{Bic}(T)$ has this property by first observing that $$\bigvee_{[a,b] \text{ an edge of $T$}} \{[a,b]\} = \text{Seg}(T).$$ Now this implies that $\mu(\emptyset, \text{Seg}(T)) = (-1)^{|\{\text{edges of $T$}\}|}$ where $\mu(-,-)$ is the M\"obius function on $\text{Bic}(T)$.
\end{remark}

%\textcolor{red}{Additionally, after the first version of our work appeared on the arXiv, M\"uhle defined }

\section*{Acknowledgements}
This project started at the 2016 combinatorics REU at the School of Mathematics, University of Minnesota, Twin Cities, and was supported by NSF RTG grant DMS-1148634. The authors would like to thank Craig Corsi, Thomas McConville, Henri M\"uhle, and Vic Reiner for their valuable advice and for useful conversations. Alexander Garver also received support from NSERC and the Canada Research Chairs program during various parts of this project.

\bibliographystyle{plain}
\bibliography{bib_Phi(Bic)}

\end{document}